\documentclass[a4paper, 11pt]{amsart}
\usepackage{amsmath,amsfonts,amssymb,amsthm,enumerate}
\usepackage[hidelinks]{hyperref}
\usepackage[utf8]{inputenc}

\newtheorem{thm}{Theorem}[section]
\newtheorem{cor}[thm]{Corollary}
\newtheorem{prop}[thm]{Proposition}
\newtheorem{lem}[thm]{Lemma}

\theoremstyle{definition}
\newtheorem{defn}[thm]{Definition}
\newtheorem{exa}[thm]{Example}
\newtheorem{exas}[thm]{Examples}
\newtheorem{rem}[thm]{Remark}

\newcommand{\defeq}{\mathrel{\mathop:}=}
\newcommand{\B}{\mathbb B}
\newcommand{\mcG}{\mathcal G}
\newcommand{\mcGo}{\mathcal G^{(0)}}

\let\on\operatorname
\let\phi\varphi

\begin{document}

\title{On Steinberg algebras of Hausdorff ample groupoids over commutative semirings}

\author{Tran Giang Nam}
\address{Institute of Mathematics, VAST \\ 18 Hoang Quoc Viet, Cau Giay, Hanoi, Vietnam}
\email{tgnam@math.ac.vn}

\author{Jens Zumbr\"agel}
\address{Faculty of Computer Science and Mathematics \\ University of Passau, Germany}
\email{jens.zumbraegel@uni-passau.de}

\thanks{The first author was partially supported by FAPESP of Brazil Proc.\ 2018/06538-6 and by the Vietnam Academy of Science and Technology grant CT0000.02/20-21.}


\begin{abstract} We investigate the algebra of a Hausdorff ample groupoid, introduced by Steinberg, over a commutative semiring~$S$. In particular, we obtain a complete characterization of congruence-simpleness for such Steinberg algebras, extending the well-known characterizations when~$S$ is a field or a commutative ring. We also provide a criterion for the Steinberg algebra $A_S(\mcG_E)$ of the graph groupoid~$\mcG_E$ associated to an arbitrary graph~$E$ to be congruence-simple. Motivated by a result of Clark and Sims, we show that the natural homomorphism from the Leavitt path algebra $L_{\B}(E)$ to the Steinberg algebra $A_{\B}(\mcG_E)$, where~$\B$ is the Boolean semifield, is an isomorphism if and only if~$E$ is row-finite. Moreover, we establish the Reduction Theorem and Uniqueness Theorems for Leavitt path algebras of row-finite graphs over the Boolean semifield~$\B$. \medskip

\end{abstract}

\maketitle

\section{Introduction}\label{sec:intro}

Steinberg algebras have been devised in \cite{s:agatdisa} in the context of discrete inverse semigroup algebras and independently in \cite{cfst:aggolpa} as a model for Leavitt path algebras.
They can be seen as discrete analogs of groupoid $C^*$-algebras, which were introduced earlier (see, \textit{e.g.}, \cite{r:agatca, p:gisatoa, e:isacca}).
The concept of a Steinberg algebra encompasses group algebras, inverse semigroup algebras and Leavitt path algebras. In recent years, there has been considerable interest around Steinberg algebras and in particular regarding their simpleness (see, \textit{e.g.}, \cite{bcfs:soaateg, ce:utfsa, s:spasoegawatisa, cepss:soaatnhg, n:soavess}).

Semirings have found their place in various branches of Mathematics, Computer Science, Physics, and other areas (see, for instance, \cite{g:sata}). There has been a substantial amount of interest in additively idempotent
semirings --- among which the Boolean semifield, tropical semifields, and coordinate semirings of tropical varieties represent prominent
examples --- originated in several emerging areas such as Tropical Geometry \cite{rst:fsitg, gg:eotv}, Tropical Algebra \cite{ir:sa},
$\mathbb F_1$-Geometry \cite{ccm:fw, cc:sofazf, cc:gotas}, the Geometry of Blueprints \cite{lor:tgob}, Cryptography \cite{mmr:pkcbosa},
Weighted automata \cite{dk:csafps}, Cluster algebras \cite{k:caadc}, Mathematical Physics \cite{l:tmdiatmavbi} and MV-algebras~\cite{dg:allatsr}.

In the development of structure theories for varieties of algebras, so-called congruence-simple algebras, \textit{i.e.}, algebras possessing only two trivial congruences -- the diagonal and universal
ones -- play a pivotal role as ``building blocks''.  In addition, some important applications of congruence-simple semirings include constructions of novel semigroup actions for a potential use in public-key cryptosystems (see, \textit{e.g.},
\cite{mmr:pkcbosa}). In this regard, a fundamental problem is therefore to classify congruence-simple semirings, in particular additively idempotent congruence-simple semirings.

Recently, there has been a number of works addressing this problem for certain special classes of semirings (see, \textit{e.g.}, \cite{mf:ccs, bhjk:scs, m:ofcss, bk:css, z:cofcsswz, kz:fsais, knz:ososacs, knz:solpawcias}). In particular, commutative congruence-simple semirings were completely classified -- they are exactly either fields or the Boolean semifield~$\B$ (see \cite{mf:ccs, bhjk:scs, bk:css}); finite congruence-simple semirings were classified in \cite{m:ofcss, z:cofcsswz, kz:fsais}; Katsov and the authors~\cite{knz:ososacs} described congruence-simple complete semirings, providing a method to construct additively idempotent congruence-simple infinite semirings by using the endomorphism semiring of semilattices; moreover, Katsov and the authors~\cite{knz:solpawcias} gave a criterion for the Leavitt path algebra of a row-finite graph over a commutative semiring to be congruence-simple, which forms a method to construct additively idempotent congruence-simple infinite semirings based on directed graphs. However, the classification of congruence-simple infinite semirings in general remains to be an important unresolved problem, on which the present paper aims to contribute.

Motivated by the constructions of~\cite{s:agatdisa} and~\cite{cfst:aggolpa}, we introduce and study the concept of Steinberg algebras of Hausdorff ample groupoids in a ``non-additive'' semiring setting, and investigate congruence-simpleness for these algebras. This semiring setup showcases interesting novel attributes of the Steinberg algebras. For example, contrary to the ring case, it turns out that an algebra of a finite inverse semigroup over a semiring is in fact not necessarily isomorphic to its associated Steinberg algebra. Also note that in our semiring setting, as opposed to the ``additive'' ring case, congruence-simpleness is not the same as ideal-simpleness, \textit{i.e.}, having only trivial ideals (see below or \cite[Ex.~3.8]{knz:ososacs}).

A main goal of this paper is to characterize congruence-simple Steinberg algebras of Hausdorff ample groupoids over a commutative ground semiring~$S$, extending the well-known characterizations when~$S$ is a field or a commutative unital ring (see \cite[Th.~4.1]{bcfs:soaateg}, \cite[Th.~4.1, Cor.~4.6]{ce:utfsa} and \cite[Th~3.5]{s:spasoegawatisa}).
Furthermore, we describe congruence-simple Steinberg algebras $A_S(\mcG_E)$ of graph groupoids $\mcG_E$ associated to graphs~$E$ over a commutative semiring $S$, and investigate the isomorphism problem between the Steinberg algebras $A_S(\mcG_E)$ and the Leavitt path algebras $L_S(E)$ when~$S$ is an additively idempotent commutative semiring. The new constructions of additively idempotent congruence-simple infinite semirings based on Hausdorff ample groupoids complement well the recent constructions of congruence-simple semirings that use the endomorphism semiring of semilattices and Leavitt path algebras of row-finite graphs with coefficients in the Boolean semifield~$\B$ mentioned above.

It should be emphasized that in the semiring setting we have to work on congruences which are different from ideals, and hence some different, novel techniques have to be applied in places. Namely, a key technique is first to reduce the problems to additively idempotent semirings and then use the natural order on additively idempotent semirings to address them. For example, Clark and Sims \cite[Ex.~3.2]{cs:eghmesa} constructed an isomorphism (called the \emph{natural isomorphism}) from the Leavitt path algebra $L_S(E)$ onto the Steinberg algebra $A_S(\mcG_E)$ when~$S$ is a commutative unital ring, by using Tomforde's Graded Uniqueness Theorem \cite[Th.~5.3]{t:lpawciacr} which is based on the theory of graded algebras and homogeneous ideals. In our semiring setting, however, concepts like homogeneous ideal and graded quotient algebra are not well-established, and so Clark and Sims's result is, in general, not true in the semiring setting --- in fact, it is only true when~$E$ is a row-finite graph if~$S$ is the Boolean semifield. We establish analogs of the Reduction Theorem \cite[Th.~2.2.11]{aas:lpa} and Tomforde's Uniqueness Theorem for the Leavitt path algebras $L_{\B}(E)$ of row-finite graphs $E$ over the Boolean semifield $\B$, using new proof techniques on congruences via the natural order.

The article is organized as follows. For the reader's convenience, all subsequently necessary basic
concepts and facts on semirings and Steinberg algebras over a commutative semiring are collected in Section~\ref{sec:basic}.
In Section~\ref{sec:simple}, we provide a complete description of congruence-simple Steinberg algebras of Hausdorff ample groupoids over a commutative semiring (Theorem~\ref{Nec-suffcondtheo}).
In Section~\ref{sec:graph}, we give a complete characterization of  congruence-simple Steinberg algebras $A_S(\mcG_E)$ of graph groupoids $\mcG_E$ associated to arbitrary graphs~$E$ over a commutative semiring~$S$ (Theorem~\ref{con-sim-gragrouSteinAlg}). Motivated by Clark and Sims's result \cite[Ex.~3.2]{cs:eghmesa}, we show that the natural homomorphism from  the Leavitt path algebra $L_{\B}(E)$ to the Steinberg algebra $A_{\B}(\mcG_E)$ is an isomorphism if and only if $E$ is row-finite (Theorem~\ref{LPAs are Steinberg-Alg}). In order to do so, we establish the Reduction Theorem (Lemma~\ref{graph-cong-reduction}) and Uniqueness Theorems (Corollaries~\ref{Uniqueness Thm}  and \ref{Cuntz-Krieger UniThm}) for Leavitt path algebras of row-finite graphs over $\B$.  Also, we show by example that the Leavitt path algebras $L_{\B}(E)$ is, in general, not isomorphic to the Steinberg algebra $A_{\B}(\mcG_E)$ (Example~\ref{LPAs are not Steinberg-Alg}). This provides us with examples of additively idempotent congruence-simple semirings by using graph groupoids, which are not isomorphic to the corresponding Leavitt path algebras (Remark~\ref{finalrem}).

\section{Basic concepts}\label{sec:basic}

\subsection{Preliminaries on semirings}

Recall~\cite{g:sata} that a \emph{hemiring} is an algebra $(S, +,
\cdot, 0)$ such that the following conditions are satisfied:
\begin{enumerate}
\item $(S, +, 0)$ is a commutative monoid with identity element~$0$;
\item $(S, \cdot)$ is a semigroup;
\item Multiplication distributes over addition from either side;
\item $0 s = 0 = s 0$ for all $s \in S$.
\end{enumerate}

A hemiring~$S$ is \emph{commutative} if $(S, \cdot)$ is a commutative
semigroup; and a hemiring $S$ is \emph{additively idempotent} if
$a + a = a$ for all $a \in S$.  Moreover, a hemiring~$S$ is a
\emph{semiring} if its multiplicative semigroup $(S, \cdot)$ actually
is a monoid $(S, \cdot, 1)$ with identity element~$1 \ne 0$.  A
commutative semiring~$S$ is a \emph{semifield} if $(S \!\setminus\!
\{ 0 \}, \cdot, 1)$ is a group.  Two well-known examples of semifields
are the additively idempotent two element semiring $\B \defeq
(\{ 0, 1\}, \vee, \wedge, 0, 1)$, the so-called \emph{Boolean semifield},
as well as the \emph{tropical semifield} $\mathbb T \defeq (\mathbb R
\cup \{ -\infty \}, \vee , +, -\infty, 0)$.

As usual, given two hemirings $S$ and $S'$, a map $\phi \colon S
\longrightarrow S'$ is a \emph{homomorphism} if $\phi(0) = 0$,
$\phi(x+y) = \phi(x) + \phi(y)$ and $\phi(x y) =
\phi(x) \phi(y)$ for all $x,y\in S$; and a submonoid~$I$ of
$(S,+,0)$ is an \emph{ideal} of a hemiring~$S$ if $sa$ and $as\in I$
for all $a \in I$ and $s \in S$; an equivalence relation $\rho$ on a
hemiring~$S$ is a \emph{congruence} if $(s\!+\!a,s\!+\!b)\in \rho$,
$(sa,sb)\in \rho $ and $(as, bs) \in \rho$ for all pairs $(a,b) \in
\rho $ and $s \in S$. On every hemiring~$S$ there are always the two
trivial congruences --- the \emph{diagonal congruence},
$\vartriangle_{_S} \,\defeq \{ (s, s) \mid s \in S \}$, and the
\emph{universal congruence}, $S^{2} \defeq \{ (a, b) \mid a, b
\in S \}$. Following~\cite{bhjk:scs}, a nonzero hemiring~$S$ is
\emph{congruence-simple} if $\vartriangle_{_S}$ and $S^{2}$ are the
only congruences on~$S$.

\begin{rem}\label{Cong-simRem}
A nonzero hemiring~$S$ is congruence-simple if and only if every nonzero hemiring homomorphism $\phi \colon S \longrightarrow S'$ is injective.
\end{rem}

\begin{proof}
($\Longrightarrow$). Assume that~$S$ is a congruence-simple hemiring and $\phi \colon S \longrightarrow S'$ is a nonzero homomorphism. We then have that the set \[ \ker(\phi) \defeq \{ (x, y) \in S^2 \mid \phi(x) = \phi(y) \} \] is a congruence on~$S$. Also, since~$\phi$ is nonzero, $\phi(x) \ne 0 = \phi(0)$ for some $x \in S$, and so $(x, 0) \notin \ker(\phi)$. From these observations, and since~$S$ is congruence-simple, we immediately obtain that $\ker(\phi) = \,\vartriangle_{_S}$. This implies that~$\phi$ is injective.

($\Longleftarrow$). Let~$\rho$ be a congruence on~$S$ which is different from the universal congruence. We then have the quotient semiring $S / \rho$ is nonzero, and the natural projection mapping $\pi \colon S \longrightarrow S/\rho$, defined by $\pi(s) = [s]$ for all $s \in S$, is a nonzero hemiring homomorphism. By our hypothesis, $\pi$ is injective, and so $\rho = \,\vartriangle_{_S}$. This implies that~$S$ is congruence-simple, finishing the proof. 
\end{proof}

We note that a ring~$R$ is congruence-simple if and only if~$\{ 0 \}$ and~$R$ are the only ideals of~$R$ (\textit{i.e.}, it is a \emph{simple} ring).  However, this is in general not true in a semiring setting.  For example, the tropical semifield~$\mathbb T$ has only the trivial ideals, but it has a proper congruence $\rho$ defined by $(x, y) \in \rho$ iff $x = y$ or $x + y \ne -\infty$, for $x, y \in \mathbb T$; that means, $\mathbb T$ is not congruence-simple.

An $S$-\emph{semimodule} over a given commutative semiring~$S$ is a
commutative monoid $(M,+,0_{M})$ together with a scalar multiplication
$(s,m) \mapsto sm$ from $S\times M$ to~$M$ which satisfies the
identities $(ss')m = s(s'm)$, $s(m+m') = sm+sm'$, $(s+s')m = sm+s'm$,
$1m=m$, $s0_{M} = 0_{M} = 0m$ for all $s,s'\in S$ and $m,m'\in
M$. \emph{Homomorphisms} between semimodules and \emph{free}
semimodules are defined in the standard manner.

By an $S$-algebra $A$ over a given commutative semiring~$S$ we mean an
$S$-semimodule~$A$ with an associative bilinear $S$-semimodule
multiplication ``\,$\cdot$\,'' on~$A$.  An $S$-algebra~$A$ is
\emph{unital} if $(A,\cdot)$ is actually a monoid with a neutral
element $1_{A}\in A$, \textit{i.e.}, $a1_{A}=a=1_{A}a$ for all $a\in
A$.  For example, every hemiring is an $\mathbb N$-algebra, where
$\mathbb N$ is the semiring of non-negative integers; and, of course,
every additively idempotent hemiring is a $\B$-algebra.
Homomorphisms between algebras over commutative semirings are defined
in the standard manner.

Let~$S$ be a commutative semiring and $\{x_i \mid i \in I\}$ a set of
independent, noncommuting indeterminates. Then $S \langle x_i \mid i \in
I \rangle$ will denote the free $S$-algebra generated by the
indeterminates $\{x_i \mid i\in I\}$, whose elements are polynomials
in the noncommuting variables $x_i$, $i \in I$, with coefficients
from~$S$ that commute with each variable.

Finally, let~$S$ be a commutative semiring and $(G, \cdot, 1)$ a
group.  Then we can form the \emph{group semiring} $S[G]$, whose
elements are formal sums $\sum_{g\in G}a_{g}g$ with
\emph{coefficients} $a_g \in S$ and finite support,
\textit{i.e.}, almost all $a_g = 0$. As usual, the operations of
addition and multiplication on $S[G]$ are defined as follows
\begin{gather*}
\sum_{g\in G} a_g g + \sum_{g\in G} b_g g = \sum_{g\in G} (a_g+b_g) g  \text{\,\, and\, \,}
(\sum_{g\in G} a_g g) (\sum_{h\in G} b_h h) = \sum_{t\in G} c_t t ,
\end{gather*}
where $c_t = \sum a_g b_h$, with summation over all $(g,h) \in G
\!\times\! G$ such that $g h = t$.  Clearly, the elements of $S \defeq
S\cdot 1$ commute with the elements of $G \defeq 1 \cdot G$ under the
multiplication in $S[G]$.  In particular, one may easily see that
$S[\mathbb Z] \cong S[x,x^{-1}]$, where $S[x,x^{-1}]$ is the algebra
of the \textit{Laurent polynomials} over~$S$.

\subsection{Steinberg algebras over commutative semirings}

In this subsection, we introduce the Steinberg algebra of a Hausdorff ample groupoid over an arbitrary commutative semiring. The construction of such an algebra is a natural generalization of the constructions of Steinberg algebras over commutative rings as introduced in \cite{s:agatdisa} in the context of discrete inverse semigroup algebras, and independently in~\cite{cfst:aggolpa} as a model for Leavitt path algebras.
All these constructions are crucially based on some general notions of groupoids that for the reader's convenience we reproduce here.

A \emph{groupoid} is a small category in which every morphism is invertible.  It can also be viewed as a generalization of a group which has a partial binary operation. 
Let~$\mcG$ be a groupoid.  If $\alpha \in \mcG$, $s(\alpha) = \alpha^{-1} \alpha$ is the \emph{source} of~$\alpha$ and $r(\alpha) = \alpha \alpha^{-1}$ is its \emph{range}.  The pair $(\alpha, \beta)$ is composable if and only if $r(\beta) = s(\alpha)$.  The set $\mcGo \defeq s(\mcG) = r(\mcG)$ is called the \textit{unit space} of $\mcG$.  Elements of $\mcGo$ are units in the sense that $\alpha s(\alpha) = \alpha$ and
$r(\alpha) \alpha = \alpha$ for all $\alpha \in \mcG$.  For $U, V \subseteq \mcG$, we define
\[ U V \defeq \{ \alpha \beta \mid \alpha \in U ,\, \beta \in V ,\,
  r(\beta) = s(\alpha) \} \quad \text{ and } \quad
  U^{-1} \defeq \{ \alpha^{-1} \mid \alpha \in U \} \,. \]

A \textit{topological groupoid} is a groupoid endowed with a topology under which the inverse map is continuous, and such that the composition is continuous with respect to the relative topology on $\mcG^{(2)} \defeq \{( \alpha, \beta ) \in \mcG^2 \mid r(\beta) = s(\alpha) \}$ inherited from $\mcG^2$. An \textit{\'etale groupoid} is a topological groupoid $\mcG,$ whose unit space  $\mcGo$ is locally compact Hausdorff, and such that 
 the domain map $s$ is a local homeomorphism. In this case, the range map $r$ and the multiplication map are local homeomorphisms and  $\mcGo$ is open in $\mcG$ \cite{r:egatq}. 

An \textit{open bisection} of $\mcG$ is an open subset $U\subseteq \mcG$ such that $s|_U$ and $r|_U$ are homeomorphisms onto an open subset of $\mcGo$.  

\begin{lem}[{\cite[Prop.~2.2.4]{p:gisatoa} and \cite[Lem.~2.1]{r:tgatlpa}}]\label{cobiseclem}
Let~$\mcG$ be an \'etale groupoid, and let $U$ and $V$ be compact open bisections of~$\mcG$. Then the following holds:
\begin{enumerate}[\quad \upshape (1)]
\item $U^{-1}$ and $UV$ are compact open bisections,
\item If $\mcG$ is Hausdorff, then $U\cap V$ is a compact open bisection.\end{enumerate}	
\end{lem}

An \'etale groupoid~$\mcG$ is called \textit{ample} if~$\mcG$ has a base of compact open bisections for its topology.

Let~$\mcG$ be a Hausdorff ample groupoid, and~$S$ a commutative semiring with discrete topology.  We denote by $S^{\mcG}$ the set of all continuous functions from~$\mcG$ to~$S$.  Canonically, $S^{\mcG}$ has the structure of an $S$-semimodule with operations defined pointwise.
Notice that for any compact open bisection~$U$ of~$\mcG$,  the function $1_U \colon \mcG \longrightarrow S$, which denotes the characteristic function of~$U$, is continuous with compact support, i.e, $1_U \in S^{\mcG}$.

\begin{defn}
Let~$\mcG$ be a Hausdorff ample groupoid, and~$S$ a commutative semiring.   Let $A_S(\mcG)$ be the $S$-subsemimodule of $S^{\mcG}$ generated by the set
\[ \{ 1_U \mid U \text{ is a compact open bisection of } \mcG \} . \]
\end{defn}

\begin{lem}\label{expresslem}
Let $\mcG$ be a Hausdorff ample groupoid with a base~$\mathcal B$ of compact open bisections, and $S$ a commutative semiring.
\begin{enumerate}[\quad \upshape (1)]
\item Every $f \in A_S(\mcG)$ can be expressed as $f = \sum_{i=1}^n s_i 1_{U_i}$, where $s_i \in S \!\setminus\! \{ 0 \}$, and $U_1, \dots, U_n$ are mutually disjoint compact open bisections of~$\mcG$.
\item If~$S$ is an additively idempotent commutative semiring, then there holds $A_S(\mcG) = \{f \in S^{\mcG} \mid f \text{ has compact support} \} = \on{Span}_S \{ 1_B \mid B \in \mathcal B \}$.
\end{enumerate}
\end{lem}

\begin{proof}
(1) We first note that $U \setminus V$ is a compact open bisection for all compact open bisections~$U$ and~$V$ of~$\mcG$.  Indeed, by Lemma~\ref{cobiseclem}, $U \cap V$ is a compact open bisection of $\mcG$, and so $U \cap V$ is clopen in $\mcG$.  It implies that $U \cap V$ is clopen in~$U$.  Then $U \setminus V = U \setminus (U \cap V)$ is a clopen subset of~$U$, and hence $U \setminus V$ is a compact open bisection of~$\mcG$.  We also note that
\[ s 1_U + r 1_V = s 1_{U \setminus V} + (s + r) 1_{U \cap V} + r 1_{V \setminus U} \]
for all $s, r \in S$, and for all compact open bisections $U, V$ of~$\mcG$.  From these notes and by induction, we immediately obtain statement (1) of the lemma. \medskip

(2) It is obvious that 
\[ \on{Span}_S \{ 1_B \mid B \in \mathcal B \} \subseteq A_S(\mcG) \subseteq \{ f \in S^{\mcG} \mid f \text{ has compact support} \}. \]
Let $f \colon \mcG \longrightarrow S$ be a continuous function with compact support.  We then have that $f(\mcG) \setminus \{ 0 \}$
is contained in a compact subset of the discrete space~$S$, and
so it is finite.  Assume that $f(\mcG) \setminus \{ 0 \} = \{ s_1, \dots, s_n\}$.  Then, each $U_i = f^{-1}(s_i)$ is compact open in $\mcG$ and $f = s_1 1_{U_1} + \ldots + s_n 1_{U_n}$.  Since~$\mathcal B$ is a base of compact open bisections for the topology on~$\mcG$, for each $1 \le i \le n$, there exist elements $B^i_1, \dots, B^i_k \in \mathcal B$ such that $U_i = B^i_1 \cup \cdots \cup B^i_k$.  Since the semiring~$S$ is additively idempotent, we immediately obtain that
\[ 1_{U_i} = 1_{\bigcup^k_{j=1} B^i_j} = 1_{B^i_1} + \ldots + 1_{B^i_k} \in \on{Span}_S \{ 1_B \mid B \in \mathcal B \} , \] so $f \in \on{Span}_S \{ 1_B \mid B \in \mathcal B \}$.  It implies that $A_S(\mcG) = \{ f \in S^{\mcG} \mid f$ has compact support$\} = \on{Span}_S \{ 1_B \mid B \in \mathcal B \}$, thus finishing the proof.
\end{proof}

We now define the convolution product on the $S$-semimodule $A_S(\mcG)$ in order to make it an $S$-algebra.

\begin{defn}[{cf.~\cite[Def.~4.4]{s:agatdisa}}]
Let~$\mcG$ be a Hausdorff ample groupoid, and~$S$ a commutative semiring.   The \textit{multiplication} of $f, g \in A_S(\mcG)$ is given, for $\gamma \in \mcG$, by the convolution \[ (f \ast g) (\gamma) \,\defeq \!\sum_{\substack{r(\beta) = s(\alpha)\\ \gamma = \alpha \beta}}\! f(\alpha) g(\beta) \,. \]
\end{defn}

One must show that this sum is really finite and $f \ast g$ belongs to $A_S(\mcG)$, which is the content of the following proposition.

\begin{prop}[{cf.~\cite[Prop.~4.5, 4.6]{s:agatdisa}}]\label{convoprod}
Let~$\mcG$ be a Hausdorff ample groupoid, and~$S$ a commutative semiring.  Then the following is true:
\begin{enumerate}[\quad \upshape (1)]
\item $f\ast g \in A_S(\mcG)$ for all $f, g\in A_S(\mcG)$;
\item $1_U\ast 1_V = 1_{UV}$ for all compact open bisections $U, V$ of $\mcG$. In particular, if $U$ and $V$ are compact open subsets of $\mcGo$, then $1_U\ast 1_V = 1_{U\cap V}$;
\item For any compact open bisection $U$ of $\mcG$, $1_{U^{-1}}(\gamma) = 1_U(\gamma^{-1})$ for all $\gamma\in \mcG$;
\item $A_S(\mcG)$, equipped with the convolution, is an $S$-algebra.
\end{enumerate}
\end{prop}

\begin{proof}
Items (1) to (3) are proved similarly as in the proof of \cite[Prop.~4.5]{s:agatdisa}. For item (4), it is sufficient to show the associativity of convolution. However, this is a straightforward by using item (2) (or the reader can refer to the proof of \cite[Prop.~2.4]{r:tgatlpa}), finishing the proof. 
\end{proof}

\begin{defn}
Let~$\mcG$ be a Hausdorff ample groupoid, and~$S$ a commutative semiring.  We call the $S$-algebra $A_S(\mcG)$ the \textit{Steinberg algebra} of~$\mcG$ over~$S$.
\end{defn} 

The following examples illustrate that some well-known algebras can be viewed as Steinberg algebras as well.

\begin{exas}
(1) Let $S$ be a commutative semiring, and $G$ a group. Define a small category $\mcG$ with one object $e$ (the identity of $G$) and $\text{Hom}_{\mcG}(e, e) = G$, where the composition of morphisms is simply the group multiplication. Then~$\mcG$ is obviously a Hausdorff ample groupoid with respect to the discrete topology, and it has a base of compact open bisections which are the singletons $\{g\}$.  The algebra $A_S(\mcG)$ is isomorphic to the group semiring $S[G]$ by the map $1_{\{g\}} \longmapsto g$. \medskip

(2) Let~$S$ be a commutative semiring and $X = \{ x_1, \dots, x_n \}$ a finite set.  Then, $\mcG \defeq X \times X$ is a groupoid with the composition and inverse defined, respectively, by $(x, y) (y, z) = (x, z)$ and $(x, y)^{-1} = (y, x)$.  Furthermore, $\mcG$ is a Hausdorff ample groupoid with respect to the discrete topology, and it has a base of compact open bisections which are the singletons $\{(x_i, x_j)\}$. In this example, $A_S(\mcG)$ is isomorphic to the $n \times n$ matrix semiring $M_n(S)$ by the map $1_{\{(x_i, x_j)\}}\longmapsto E_{ij}$, where $\{E_{ij} \mid 1 \le i,j \le n\}$ are the matrix units in $M_n(S)$. \medskip

$(3)$ Let $S$ be a commutative semiring and $\mcG$ a discrete groupoid. It is not hard to see that
\[ 1_{\{g\}} \ast 1_{\{h\}} = \begin{cases}
    1_{\{gh\}} & \text{if } r(h) =  s(g) , \\
    0 & \text{otherwise}, \end{cases} \]
for all $g, h\in \mcG$. Then $A_S(\mcG)$ is exactly the $S$-algebra having basis $\mcG$ and whose product extends that of $\mcG$ where we interpret undefined products as $0$.
\end{exas}

\begin{rem}
Let~$\mcG$ be a Hausdorff ample groupoid and let~$S$ be an additively idempotent commutative semiring.
  
(1) For all compact open subsets~$U$, $V$ of~$\mcG$ there holds \[ 1_U * 1_V = 1_{U V} \,, \] extending Proposition~\ref{convoprod} (2).
Indeed, by additive idempotency, for $\gamma \in \mcG$ the value $(1_U * 1_V)(\gamma) = \sum_{\gamma = \alpha \beta} 1_U(\alpha) 1_V(\beta)$ equals~$1$ iff there exists $\alpha \in U$ and $\beta \in V$ with $r(\beta) = s(\alpha)$ and $\alpha \beta = \gamma$, \text{i.e.}, iff $\gamma \in U V$.

(2) If moreover the semiring~$S$ is the Boolean semifield~$\B$, there is a natural bijective correspondence between the Steinberg algebra $A_{\B}(\mcG)$ and the collection of all compact open subsets of~$\mcG$, given by $1_U \mapsto U$ where $U \subseteq \mcG$ is a compact open subset, cf.\ Lemma~\ref{expresslem} (2).
Under this bijection the Steinberg algebra operations correspond to the set operations \[ U + V \defeq U \cup V \,, \qquad U * V \defeq U V \] for any compact open subsets~$U$, $V$ of~$\mcG$.
\end{rem}

As usual, for a hemiring~$S$ a \emph{set of local units} is a set $F \subseteq S$ of idempotents in~$S$ such that, for every finite subset $\{s_1, \dots, s_n\} \subseteq S$, there exists an element $f \in F$ with $f s_i = s_i = s_i f$ for all $i = 1, \dots, n$.  Using Proposition~\ref{convoprod} and repeating verbatim the proofs of \cite[Prop.~4.11]{s:agatdisa} and \cite[Lem.~2.6]{cep:agutatgisosa}, one obtains the following useful fact.

\begin{prop}\label{unitalSteinAlg}
Let~$\mcG$ be a Hausdorff ample groupoid, and~$S$ a commutative semiring. Then the following holds:
\begin{enumerate}[\quad \upshape (1)]
\item $(${cf.~\cite[Prop.~4.11]{s:agatdisa}}$)$ The algebra $A_S(\mcG)$ is unital if and only if $\mcGo$ is compact; in this case, the identity element is $1 = 1_{\mcGo}$.
\item $(${cf.~\cite[Lem.~2.6]{cep:agutatgisosa}}$)$ A set of local units of $A_S(\mcG)$ is given by $\{1_U \mid U$ is a compact open subset of~$\mcGo\}$.
\end{enumerate}
\end{prop}

We next present the universal property of Steinberg algebras over commutative semirings, whose proof is completely analogous to the one in \cite[Th.~3.10]{cfst:aggolpa} and which, for the reader's convenience, we provide here.

\begin{thm}[{cf. \cite[Th.~3.10]{cfst:aggolpa}}]\label{univproperty}
Let $\mcG$ be a Hausdorff ample groupoid, and $S$ a commutative semiring. Let $B$ be an $S$-algebra containing a family of elements $\{t_U\mid U$ is a compact open bisection of $\mcG\}$ satisfying:
\begin{enumerate}[\quad \upshape (1)]
\item $t_{\varnothing} = 0$;
\item $t_U t_V = t_{U V}$ for all compact open bisections~$U$ and~$V$; and
\item $t_U + t_V = t_{U \cup V}$ whenever~$U$ and~$V$ are disjoint compact open bisections such that $U \cup V$ is a bisection.
\end{enumerate}
Then, there is a unique $S$-algebra homomorphism $\pi \colon A_S(\mcG) \longrightarrow B$ such that $\pi(1_U) = t_U$ for all compact open bisections~$U$.
\end{thm}

\begin{proof}
First we observe that, by condition (3) and induction on~$n$, there holds $t_{\bigcup^n_{i=1} U_i} = \sum^n_{i=1} t_{U_i}$ whenever $U_1, \dots, U_n$ are mutually disjoint compact open bisections such that $\bigcup^n_{i=1} U_i$ is a (compact open) bisection.

Next we show that the formula $\sum_{U\in F}a_U1_U\longmapsto \sum_{U\in F}a_Ut_U$ is well-defined on linear combinations of indicator functions, where $F$ is a finite set of mutually disjoint compact open bisections. Assume that
\[ \sum_{U \in F} a_U 1_U = \sum_{V \in H} b_V 1_V , \]
where each of~$F$ and~$H$ is a finite set of mutually disjoint compact open bisections.  Let $K = \{ U \cap V \mid U \in F ,\, V \in H ,\, U \cap V \ne \varnothing \}$.  Then, since~$\mcG$ is Hausdorff, and by Lemma~\ref{cobiseclem}, every element of~$K$ is a compact open bisection of~$\mcG$.  Also, for each $U \in F$ and each $V \in H$, we have that
$U = \bigsqcup \{ W \in K \mid W \subseteq U \}$ and $V = \bigsqcup \{ W \in K \mid W \subseteq V \}$, so $t_U = \sum_{W \in K ,\, W\subseteq U} t_W$ and $t_V = \sum_{W \in K ,\, W\subseteq V} t_W$. Hence we find that
\[ \sum_{U \in F} a_U t_U = \sum_{U \in F} \sum_{W \in K ,\, W \subseteq U} a_U t_W = \sum_{W \in K} \big( \sum_{U \in F ,\, W \subseteq U} a_U \big) t_W , \]
and similarly
\[\sum_{V \in H} b_V t_V = \sum_{W \in K} \big( \sum_{V \in H ,\, W \subseteq V} b_V \big) t_W . \]
Fix $W \in K$ and let $\alpha \in W$.  By definition of~$K$, for all $U \in F$, we obtain that $\alpha \in U$ if and only if $W \subseteq U$.  Therefore,
\[ \sum_{U \in F} a_U 1_U (\alpha) = \sum_{U \in F ,\, \alpha \in U} a_U = \sum_{U \in F ,\, W \subseteq U} a_U . \]
Similarly, $\sum_{V \in H} b_V 1_V(\alpha) = \sum_{V \in H ,\, W \subseteq V} b_V$, and hence
\[ \sum_{U \in F ,\, W \subseteq U} a_U = \sum_{V \in H ,\, W \subseteq V} b_V . \]
It follows that $\sum_{U \in F} a_U t_U = \sum_{V \in H} b_V t_V$.  So there exists an $S$-homomorphism $\pi \colon A_S(\mcG) \longrightarrow B$ such that $\pi(1_U) = t_U$ for all compact open bisections~$U$.  It is sufficient to show that~$\pi$ is multiplicative.  However, this is straightforward by using Proposition~\ref{convoprod} (2), finishing the proof. 
\end{proof}

A major motivation for introducing Steinberg algebras has been the
study of discrete inverse semigroup algebras~\cite{s:agatdisa}.  An
\emph{inverse semigroup} is a semigroup~$G$ such that for each
$a \in G$ there is a unique $b \in G$ (denoted~$a^*$) satisfying
$a b a = a$ and $b a b = b$.  The idempotent elements~$E_G$ in an
inverse semigroup~$G$ form a commutative idempotent semigroup,
\textit{i.e.}, a semilattice.  Moreover, any inverse semigroup~$G$
defines a groupoid~$\mcG_G$ by letting the unit space be $E_G$ and
interpreting each $a \in G$ as an invertible morphism from $s(a)
\defeq a^* a$ to $r(a) \defeq a a^*$, hence $a^{-1} = a^*$ and a
composition $a b$ is defined in~$\mcG_G$ iff $b b^* = a^* a$.

It is shown that, in particular, for any commutative unital ring~$R$
and any finite inverse semigroup~$G$ there is an $R$-algebra
homomorphism \begin{equation}\label{eq:iso}
  R[G] \,\cong\, A_R(\mcG_G) \tag{$\dagger$} \end{equation}
between the semigroup algebra of~$G$ over~$R$ and the Steinberg
algebra of the groupoid~$\mcG_G$; this isomorphism is established
using the Möbius function on the semilattice~$E_G$.  Notice that in
the special case that the inverse semigroup is itself a semilattice
$G = E$, the groupoid $\mcG_E$ consists just of units and we easily
see that $A_R(\mcG_E) \cong R^E$.  As our last result in this section
shows, the isomorphism~(\ref{eq:iso}) fails in the general semiring
setup, the reason for which may be attributed to a ``lack of zero
sums'' in general, thus illustrating an interesting new feature.

\begin{prop}\label{prop:inverse_sg}
  Let~$E$ be a finite semilattice with $\vert E \vert > 1$, and
  let~$S$ be an additively idempotent semifield.  Then the semigroup
  algebra $S[E]$ is not isomorphic to the Steinberg algebra
  $A_S(\mcG_E) \cong S^E$.
\end{prop}

\begin{proof}
  First, notice that an additively idempotent semiring~$S$ is zero-sum
  free, \textit{i.e.}, $s + t = 0$ implies $s = t = 0$, for any
  $s, t \in S$.
  
  Now suppose that an isomorphism $S[E] \cong S^E$ exists.  Then the
  semigroup algebra $S[E]$ has an identity $1 = \sum_{w \in E} s_w w
  \in S[E]$, where $s_w \in S$ for $w \in E$.  Consider the
  semigroup~$E$ as a finite meet-semilattice and let $u \in E$ be any
  maximal element.  Then $1 u = \sum_{w \in E} s_w w u = u$.  But for
  all $w \in E$ with $w \ne u$ we have $w u = w \wedge u < u$, since
  otherwise $w > u$ whereas~$u$ is maximal.  Therefore, $\sum_{w \ne u}
  s_w w u = 0$ and hence $s_w = 0$ for all $w \ne u$, as the
  semiring~$S$ is zero-sum free.  It follows that there cannot be
  distinct maximal elements in~$E$, whence by finiteness the
  meet-semilattice~$E$ has a greatest element.

  We may therefore assume that the semigroup~$E$ has a neutral
  element~$e$, and hence the semigroup algebra~$S[E]$ has an identity
  $1 = e$.  As $\vert E \vert > 1$ the isomorphism $S[E] \cong S^E$
  implies that there are idempotents $f, g \in S[E] \setminus \{ 1 \}$
  such that $f + g = 1$.  Writing $f = \sum_{w \in E} s_w w$ and
  $g = \sum_{w \in E} t_w w$, for all $w \ne e$ it follows that
  $s_w + t_w = 0$ and hence $s_w = t_w = 0$, since the semiring~$S$ is
  zero-sum free.  We infer that $f = s_e e$ and $g = t_e e$ with
  $s_e, t_e \in S$ multiplicatively idempotent.  But then, because~$S$
  is a semifield, we have $s_e, t_e \in \{ 0, 1_S \}$ so that
  $f, g \in \{ 0, 1 \}$.  This contradiction concludes the proof.
\end{proof}

\section{Congruence-simpleness of Steinberg algebras}\label{sec:simple}

The main goal of this section is to present a description of the congruence-simple Steinberg algebras $A_S(\mcG)$ of Hausdorff ample groupoids $\mcG$ over a commutative semiring~$S$, which extends the well-known description when
the ground commutative semiring~$S$ is either a field or a commutative unital ring (see \cite[Th.~4.1]{bcfs:soaateg}, \cite[Th.~4.1, Cor.~4.6]{ce:utfsa} and \cite[Th.~3.5]{s:spasoegawatisa}).  

We begin by recalling some important notations of groupoids.
Let~$\mcG$ be a groupoid and let $D$, $E$ be subsets of $\mcGo$.  Define
\[ \mcG_D \defeq \{ \gamma \in \mcG \mid s(\gamma) \in D \} \,,\ \mcG^E \defeq \{\gamma \in \mcG \mid r(\gamma) \in E \} ~\text{ and }~ \mcG_D^E \defeq \mcG_D \cap \mcG^E \,. \]
In a slight abuse of notation, for $u, v \in \mcGo$ we denote $\mcG_u \defeq \mcG_{\{ u \}}$, $\mcG^v \defeq \mcG^{\{ v \}}$ and $\mcG_u^v \defeq \mcG_u \cap \mcG^v$.  For a unit~$u$ of $\mcG$ the group $\mcG_u^u = \{ \gamma \in \mcG \mid s(\gamma) = u = r(\gamma) \}$ is called its \textit{isotropy group}.  The \textit{isotropy subgroupoid} of~$\mcG$ is $\on{Iso}(\mcG) \defeq \bigcup_{u \in \mcGo} \mcG^u_u$.  A subset~$D$ of~$\mcGo$ is called \textit{invariant} if $s(\gamma) \in D$ implies $r(\gamma) \in D$ for all $\gamma \in \mcG$; equivalently, $D = \{ r(\gamma) \mid s(\gamma)\in D\} = \{s(\gamma) \mid r(\gamma) \in D\}$.  Also, $D$ is invariant if and only if its complement is invariant.

\begin{defn}[{\cite[Def.~2.1]{bcfs:soaateg}}]
Let~$\mcG$ be a Hausdorff ample groupoid.  We say that~$\mcG$ is \textit{minimal} if~$\mcGo$ has no nontrivial open invariant subsets, and we call~$\mcG$ \textit{effective} if the interior of $\on{Iso}(\mcG) \setminus \mcGo$ is empty.
\end{defn}

Note that effective groupoids are related to so-called “topologically principal” groupoids, \textit{i.e.}, in which the units with trivial isotropy are dense in the unit space.
Any Hausdorff ample groupoid being topologically principal is in fact effective, while the converse holds if the groupoid is second-countable (see~\cite[Lem.~3.1]{bcfs:soaateg}).

We now describe necessary conditions for Steinberg algebras of Hausdorff ample groupoids over commutative semirings to be congruence-simple.

\begin{prop}\label{Neccondprop}
If the Steinberg algebra $A_S(\mcG)$ of a Hausdorff ample groupoid~$\mcG$ over a commutative semiring~$S$ is congruence-simple, then there holds:
\begin{enumerate}[\quad \upshape (1)]
\item $S$~is either a field or the Boolean semifield $\B$;
\item $\mcG$~is both minimal and effective.
\end{enumerate}
\end{prop}

\begin{proof}
(1) First, let us show that there are only the two trivial congruences on~$S$.  Indeed, if~$\sim$ is a proper congruence
on~$S$, the natural surjection $\pi \colon S \longrightarrow \overline{S} \defeq S / {\sim}$, defined by $\pi(\lambda) = \overline{\lambda}$, is neither zero nor an injective homomorphism.  For any compact open bisection~$U$ of~$\mcG$, we denote by~$t_U$ the characteristic function $\mcG \longrightarrow \overline{S}$ of~$U$.  It is clear that the collection $\{ t_U \mid U$ is a compact open bisection$\}$ of elements of $A_{\overline{S}}(\mcG)$ satisfies conditions (1), (2) and (3) of Theorem~\ref{univproperty}.  Accordingly, there is a unique $S$-algebra homomorphism $\phi \colon A_S(\mcG) \longrightarrow A_{\overline{S}}(\mcG)$ such that $\phi(\lambda 1_U) = \overline{\lambda} t_U$ for any compact open bisection~$U$ and $\lambda \in S$.  Since~$\pi$ is not injective, there exist two distinct elements $a, b \in S$ such that $\overline{a} = \overline{b}$.  Fix a nonempty compact open bisection~$U$.  We then have $a 1_U \ne b 1_U$ and $\phi(a 1_U) = \overline{a} t_U = \overline{b} t_U = \phi(b 1_U)$, and so~$\phi$ is not injective.  Therefore, $A_S(\mcG)$ is not congruence-simple by Remark~\ref{Cong-simRem}.  Thus, the commutative semiring~$S$ is congruence-simple, and it follows by \cite[Th.~3.2]{mf:ccs} that~$S$ is either a field or the Boolean semifield~$\B$. \medskip

(2) Assume that~$\mcGo$ contains a nontrivial open invariant subset~$V$.  Let $D \defeq \mcGo \setminus V$.  We then have that~$\mcG_D$ coincides with the \textit{restriction} \[ \mcG|_{D} \defeq \{ \gamma \in \mcG \mid s(\gamma), \, r(\gamma) \in D \} \] of~$\mcG$ to~$D$.
Thus~$\mcG_D$ is a topological subgroupoid of~$\mcG$ with the relative topology, and its unit space is~$D$.  Since~$D$ is closed in~$\mcGo$, and the map $s \colon \mcG \longrightarrow \mcGo$ is continuous, $\mcG_D = s^{-1}(D)$ is closed in $\mcG$, and so $U \cap \mcG_D$ is a compact open bisection of~$\mcG_D$ for any compact open bisection~$U$ of~$\mcG$.  This implies that~$\mcG_D$ is a Hausdorff ample groupoid.

For any compact open bisection~$U$ of~$\mcG$, we denote by~$t_U$ the characteristic function $\mcG_D \longrightarrow S$ of $U \cap \mcG_D$.  It is obvious that the collection $\{ t_U \mid U$ is a compact open bisection$\}$ of elements of $A_{S}(\mcG_D)$ satisfies conditions (1), (2) and (3) of Theorem~\ref{univproperty}, by which there is a unique $S$-algebra homomorphism $\phi \colon A_S(\mcG) \longrightarrow A_{S}(\mcG_D)$ such that $\phi(\lambda 1_U) = \lambda t_U$ for all compact open bisections~$U$ and $\lambda \in S$.  Since~$V$ is a nontrivial open subset of~$\mcGo$, there exist nonempty compact open subsets~$U_1$ and~$U_2$ of~$\mcGo$ such that $U_1 \subseteq V$ and $U_2 \cap D \ne \varnothing$.  This implies that $\phi(1_{U_1}) = 1_{U_1\cap \mcG_D} = 0$ (since $U_1 \cap \mcG_D = \varnothing$) and $\phi(1_{U_2}) = 1_{U_2\cap \mcG_D} \ne 0$, so~$\phi$ is a nonzero homomorphism, but not injective.  Therefore, $A_S(\mcG)$ is not congruence-simple by Remark~\ref{Cong-simRem}, whence~$\mcG$ is minimal. \medskip

We next show that~$\mcG$ is effective, following essentially the proof of \cite[Prop.~4.4]{bcfs:soaateg}.  Denote by $F(\mcGo)$ the free $S$-semimodule with basis~$\mcGo$.  Let~$U$ be a compact open bisection of~$\mcG$.
Observe that $s(\alpha) \longmapsto r(\alpha)$ determines a homeomorphism from $s(U)$ to $r(U)$.  We define a map $f_U \colon \mcGo \longrightarrow F(\mcGo)$ by
\[ f_U(x) = \begin{cases}
r(\alpha) & \text{if } x = s(\alpha) \text{ and } \alpha \in U , \\
0 & \text{otherwise}, \end{cases} \]
for all $x \in \mcGo$.  By the universal property
of the free $S$-semimodule $F(\mcGo)$, there exists an element $t_U \in \on{End}_S(F(\mcGo))$ extending~$f_U$.  Now we check that
\begin{enumerate}
\item $t_{\varnothing} = 0$;
\item $t_Ut_V = t_{UV}$ for all compact open bisections~$U$ and~$V$; and
\item $t_U + t_V = t_{U\cup V}$ whenever~$U$ and~$V$ are disjoint compact open bisections such that $U \cup V$ is a bisection.
\end{enumerate}
It is easy to see that each of these conditions holds for the functions~$f_U$, and so for the endomorphisms~$t_U$ as well.  Then, by Theorem~\ref{univproperty}, there is a unique $S$-algebra homomorphism $\phi \colon A_S(\mcG) \longrightarrow \on{End}_S(F(\mcGo))$ such that $\phi(1_U) = t_U$ for all compact open bisection~$U$.  The homomorphism~$\phi$ is nonzero because~$t_U$ is nonzero for any nonempty compact open bisection~$U$ of~$\mcG$.

Now suppose that~$\mcG$ is not effective.  By \cite[Lem.~3.1]{bcfs:soaateg}, there exists a nonempty compact open bisection $U \subseteq \mcG \setminus \mcGo$ such that $s(\alpha) = r(\alpha)$ for all $\alpha \in U$.  It implies that $U \ne s(U)$ and $t_U = t_{s(U)}$, and hence $1_U \ne 1_{s(U)}$ and \[ \phi(1_U) = t_U = t_{s(U)} = \phi(1_{s(U)}) , \] showing that~$\phi$ is not injective.  Therefore, $A_S(\mcG)$ is not congruence-simple by Remark~\ref{Cong-simRem}.  Thus, $\mcG$~is effective, finishing the proof.
\end{proof}

The following result, being a “congruence” analog of \cite[Lem.~4.2]{bcfs:soaateg} and \cite[Prop.~3.3]{s:spasoegawatisa}, plays an important role in the proof of our main result below.

\begin{lem}\label{cong-reduction}
Let~$\mcG$ be an effective Hausdorff ample groupoid and~$S$ an additively idempotent semiring.  Then for every congruence~$\rho$ on $A_S(\mcG)$ different from the diagonal congruence, there exists $s \in S \!\setminus\! \{ 0 \}$ and a nonempty compact open subset~$W$ of~$\mcGo$ such that $(s 1_W, 0) \in \rho$.
\end{lem}

\begin{proof}
Since~$\rho$ is different from the diagonal congruence, there are elements $f, g \in A_S(\mcG)$ such that $f \ne g$ and $(f, g) \in \rho$.  It is clear that $A_S(\mcG)$ is an additively idempotent hemiring, hence $A_S(\mcG)$ is partially ordered  by defining $a \le b$ iff $a + b = b$, for $a, b \in A_S(\mcG)$.
Now $(f, f \!+\! g) = (f \!+\! f, f \!+\! g) \in \rho$ and $(g, f \!+\! g) = (g \!+\! g, f \!+\! g) \in \rho$, and since $f \ne g$, either $f < f + g$ or $g < f + g$.  Therefore,  without loss of generality, one may assume that $f < g$.  Then, there exists an element $\alpha \in \mcG$ such that $f(\alpha) < g(\alpha)$.  Let~$U$ be a compact open bisection of~$\mcG$ containing $\alpha^{-1}$, and let $x \defeq r(\alpha) = \alpha \alpha^{-1} \in \mcGo$.  We have that
\[ f \ast 1_U(x) = \sum_{x = \gamma \beta} f(\gamma) 1_U(\beta) = f(\alpha \alpha^{-1} \alpha) 1_U(\alpha^{-1}) = f(\alpha) \,, \]
since the unique $\beta \in U$ such that $s(\beta) = s(x) = x$ is given by $\beta = \alpha^{-1}$.
Similarly, $g \ast 1_U(x) = g(\alpha)$, and so $f \ast 1_U(x) < g \ast 1_U(x)$.

Thus, there exist two elements $\vartheta, \psi \in A_S(\mcG)$ such that $(\vartheta, \psi) \in \rho$ with $\vartheta \le \psi$ and $\vartheta|_{\mcGo} < \psi|_{\mcGo}$ (we may take $\vartheta = f \ast 1_U$ and $\psi = g\ast 1_U$ as above).  Write 
\[ \vartheta = \sum_{i=1}^n s_i 1_{U_i} \quad\text{ and }\quad \psi = \sum_{j=1}^m t_j 1_{V_j} \]
where $s_i, t_j \in S \!\setminus\! \{ 0 \}$ and each of $\{ U_i \mid i = 1, \dots, n \}$ and $\{ V_j \mid j = 1, \dots, m \}$ is a set of mutually disjoint compact open bisections of~$\mcG$.  Since~$\mcGo$ is clopen in~$\mcG$ by the Hausdorff property, $U_i \cap \mcGo$ and $V_j \cap \mcGo$ are compact open subsets of~$\mcGo$; that means, they are compact open bisections of~$\mcG$, and so, by Lemma~\ref{cobiseclem}, $U_i \setminus (U_i \cap \mcGo)$ and $V_j \setminus (V_j \cap \mcGo)$ are also compact open bisections of~$\mcG$.  We then have that
\[ \vartheta = \sum_{i=1}^n s_i 1_{U_i \setminus (U_i \cap \mcGo)} + \sum_{i=1}^n s_i 1_{U_i \cap \mcGo} \text{ and }
\psi = \sum_{j=1}^m t_j 1_{V_j \setminus (V_j \cap \mcGo)} + \sum_{j=1}^m t_j 1_{V_j \cap \mcGo} \]
and since $\vartheta \le \psi$ and $\vartheta|_{\mcGo} < \psi|_{\mcGo}$, it follows that
\[ \bigcup_{i=1}^n U_i \setminus (U_i \cap \mcGo) \subseteq \bigcup_{j=1}^m V_j \setminus (V_j \cap \mcGo) \quad\text{and}\quad \bigcup_{i=1}^n U_i \cap \mcGo \subset \bigcup_{j=1}^m V_j \cap \mcGo . \]
The set $\bigcup_{i=1}^n U_i \cap \mcGo$ is closed in~$\mcGo$ as a compact subset in a Hausdorff space, and moreover the $V_j \cap \mcGo$ are open in $\mcGo$.  Thus $(V_j \cap \mcGo) \setminus (\bigcup_{i=1}^n U_i \cap \mcGo)$ is open in~$\mcGo$, and non-empty for some~$j_0$.
Let $K \defeq \bigcup_{j=1}^m V_j \setminus (V_j \cap \mcGo) \subseteq \mcG \setminus \mcGo$ which is compact open.  Then, because~$\mcG$ is effective, and by \cite[Lem.~3.1]{bcfs:soaateg}, there is a nonempty open subset $W \subseteq (V_{j_0} \cap \mcGo) \setminus (\bigcup_{i=1}^n U_i \cap \mcGo)$ such that $W K W = \varnothing$.  Since~$\mcGo$ has a base of compact open sets, we may assume that~$W$ is compact.  We have
\begin{align*}
1_W \ast \vartheta \ast 1_W &= \sum_{i=1}^n s_i 1_{W (U_i \setminus (U_i \cap \mcGo)) W} + s_i 1_{W \cap (\bigcup_{i=1}^n U_i \cap \mcGo) \cap W} = 0 , \\
1_W \ast \psi \ast 1_W &= \sum_{j=1}^m t_j 1_{W (V_j \setminus (V_j \cap \mcGo)) W} + t_j 1_{W \cap (\bigcup_{j=1}^m V_j \cap \mcGo) \cap W} = t_{j_0} 1_W .
\end{align*}
Since $(\psi, \vartheta) \in \rho$, and~$\rho$ is a congruence on $A_S(\mcG)$, we obtain that
\[ (t_{j_0} 1_W, 0) = (1_W \ast \psi \ast 1_W ,\, 1_W \ast \vartheta \ast 1_W) \in \rho \] with $t_{j_0} \in S \!\setminus\! \{ 0 \}$, thus finishing the proof.
\end{proof}

The following result, being an $\B$-algebra analog of \cite[Prop.~4.5]{bcfs:soaateg} and \cite[Prop.~3.4]{s:spasoegawatisa}, provides a criterion for minimal
Hausdorff ample groupoids.  It plays an important role in the proof of the subsequent main results (Theorems~\ref{Nec-suffcondtheo} and~\ref{con-sim-gragrouSteinAlg}).

\begin{lem}\label{minimalcriterion}
A Hausdorff ample groupoid~$\mcG$ is minimal if and only if~$1_V$ generates $A_{\B}(\mcG)$ as an ideal for all nonempty compact open subsets~$V$ of~$\mcGo$.
\end{lem}

\begin{proof}
($\Longrightarrow$). Assume that~$\mcG$ is a minimal Hausdorff ample groupoid.  Let~$V$ be a nonempty compact open subset of~$\mcGo$, and~$I$ an ideal of $A_{\B}(\mcG)$ generated by~$1_V$.  If~$U$ is any compact open bisection of~$\mcG$, we must prove that $1_U \in I$.  Let $K \defeq r(U) \subseteq \mcGo$.  We then have that~$K$ is a compact open subset of~$\mcGo$.  Since $s(\mcG^V)$ is a nonempty open invariant set, and~$\mcG$ is minimal, $s(\mcG^V) = \mcGo$, and hence $K \subseteq s(\mcG^V)$.  Thus, for any $u \in K$, there exists an element $\alpha_u \in \mcG$ such that $s(\alpha_u) = u$ and $r(\alpha_u) \in V$.  For each $u \in K$, let $B_u$ be a compact open bisection of~$\mcG$ containing~$\alpha_u$ such that $r(B_u) \subseteq V$ and $s(B_u) \subseteq K$.  We then have $1_{s(B_u)} = 1_{B^{-1}_u} \ast 1_V \ast 1_{B_u} \in I$.  Since~$K$ is compact, there exists a finite subset $\{ u_1, \dots, u_n \}$ of~$K$ such that $\{ s(B_{u_i}) \mid i = 1, \dots, n \}$ covers~$K$.  By Lemma~\ref{expresslem} (2), $1_K = \sum_{i=1}^n 1_{s(B_{u_i})} \in I$, and so $1_U = 1_K \ast 1_U \in I$.  It follows that $I = A_{\B}(\mcG)$. \medskip

($\Longleftarrow$). Suppose that~$\mcG$ is not minimal.  Let~$U$ be a nontrivial open invariant subset of~$\mcGo$.  Let $D \defeq \mcGo \setminus V$.  We then have that~$\mcG_D$ coincides with the restriction
$\mcG|_{D} \defeq \{ \gamma \in \mcG \mid s(\gamma), \, r(\gamma) \in D\}$
of~$\mcG$ to~$D$.  As was shown in the proof of Proposition~\ref{Neccondprop}, $\mcG|_{D}$ is a Hausdorff ample groupoid with unit space~$D$, and there is a nonzero $\B$-algebra homomorphism $\phi \colon A_{\B}(\mcG) \longrightarrow A_{\B}(\mcG_D)$ such that $\phi(1_B) = 1_{B \cap \mcG_D}$ for any compact open bisection~$B$ of~$\mcG$.  This implies that $\on{Ker}(\phi) \defeq \phi^{-1}(0)$ is a proper ideal of $A_{\B}(\mcG)$.  Since~$U$ is a nontrivial open subset of~$\mcGo$, there exists a nonempty compact open subset~$V$ of~$\mcGo$ such that $V \subseteq U$.  We then have $\phi(1_{V}) = 1_{V \cap \mcG_D} = 0$ (since $V \cap \mcG_D = \varnothing$), and so $1_V \in \on{Ker}(\phi) \setminus \{ 0 \}$, and hence, by our hypothesis, $\on{Ker}(\phi) = A_{\B}(\mcG)$, a contradiction, thus finishing the proof.
\end{proof}

We are now in position to provide the main result of this section, being a “semiring” analog of \cite[Th.~4.1]{bcfs:soaateg}, \cite[Th.~4.1, Cor.~4.6]{ce:utfsa} and \cite[Th.~3.5]{s:spasoegawatisa},
characterizing the congruence-simple Steinberg algebras of Hausdorff ample groupoids over commutative semirings.

\begin{thm}\label{Nec-suffcondtheo}
The Steinberg algebra $A_S(\mcG)$ of a Hausdorff ample groupoid~$\mcG$ over a commutative semiring~$S$ is congruence-simple if and only if the following conditions are satisfied:
\begin{enumerate}[\quad \upshape (1)]
\item $S$~is either a field or the Boolean semifield~$\B$;
\item $\mcG$~is both minimal and effective.
\end{enumerate}
\end{thm}

\begin{proof}
($\Longrightarrow$). It follows from Proposition~\ref{Neccondprop}.

($\Longleftarrow$). If~$S$ is a field, then the statement follows from \cite[Th.~3.5]{s:spasoegawatisa}.  Consider the case when $S = \B$, and let~$\rho$ be a congruence on $A_{\B}(\mcG)$ which is different from the diagonal congruence.  Since~$\mcG$ is effective, by Lemma~\ref{cong-reduction}, there exists a nonempty compact open subset~$U$ of~$\mcGo$ such that $(1_U, 0) \in \rho$.
Let us consider the ideal of $A_{\B}(\mcG)$ defined as follows: \[ I \defeq \{ f \in A_{\B}(\mcG) \mid (f, 0) \in \rho \} . \]
From the observation above, $I$~contains a nonzero element~$1_U$ with $\on{supp}(1_U) = U \subseteq \mcGo$.  By Lemma~\ref{minimalcriterion}, $I = A_{\B}(\mcG)$.  It immediately follows that $\rho = A_{\B}(\mcG)^2$, whence $A_{\B}(\mcG)$ is congruence-simple, thus finishing the proof.
\end{proof} \pagebreak

\section{Steinberg algebras of graph groupoids}\label{sec:graph}

In this section we investigate Steinberg algebras $A_{S}(\mcG_E)$ of graph groupoids~$\mcG_E$ associated to arbitrary directed graphs~$E$, over a commutative semiring~$S$. We provide a complete characterization of the congruence-simple Steinberg algebras $A_S(\mcG_E)$ (Theorem~\ref{con-sim-gragrouSteinAlg}). Motivated by Clark and Sims's result \cite[Ex.~3.2]{cs:eghmesa}, we present a criterion for the natural homomorphism from the Leavitt path algebra $L_{\B}(E)$ to the Steinberg algebra $A_{\B}(\mcG_E)$
to be an isomorphism, where~$\B$ is the Boolean semifield (Theorem~\ref{LPAs are Steinberg-Alg}). In order to do so, we establish Uniqueness Theorems for Leavitt path algebras of row-finite graphs over $\B$ (Corollaries~\ref{Uniqueness Thm} and \ref{Cuntz-Krieger UniThm}). Furthermore, we argue that the Leavitt path algebra $L_{\B}(E)$ is in general not isomorphic to the Steinberg algebra $A_{\B}(\mcG_E)$ (Example~\ref{LPAs are not Steinberg-Alg}). All these constructions are crucially based on some general notions of graph theory, which for the reader's convenience we reproduce here.

A (directed) graph $E = (E^0, E^1, s, r)$ consists of two disjoint sets~$E^0$ and~$E^1$, called \emph{vertices} and \emph{edges} respectively, together with two maps $s, r \colon E^1 \longrightarrow E^0$. The vertices $s(e)$ and $r(e)$ are referred to as the \emph{source} and the \emph{range}
of the edge~$e$, respectively. A graph~$E$ is called \emph{row-finite} if $|s^{-1}(v)| < \infty$ for all $v \in E^0$.
A vertex~$v$ for which $s^{-1}(v)$ is empty is called a \emph{sink}; a vertex~$v$ is \emph{regular} if $0 < |s^{-1}(v)| < \infty$; and a vertex~$v$ is an \textit{infinite emitter} if $|s^{-1}(v)| = \infty$.

A \emph{path} in a graph~$E$ is a sequence $p = e_1 \dots e_n$  of edges $e_1, \dots, e_n$ with $r(e_i) = s(e_{i+1})$ for all $1 \le i \le n \!-\! 1$. We then say that the path~$p$ starts at the vertex $s(p) \defeq s(e_1)$, ends at the vertex $r(p) \defeq r(e_n)$, and write $|p| \defeq n$ for its length. The vertices in~$E^0$ are considered to be paths of length~$0$. We denote by~$E^{*}$ the set of all paths in $E$.
A path $p$ of positive length is a \textit{closed path} based at the vertex $v$ if $s(p) = r(p) =v$.
 A \textit{cycle} based at~$v$ is a closed path $p = e_1 \dots e_n$ based at~$v$ for which the vertices $s(e_1), \dots, s(e_n)$ are distinct.
An \textit{infinite path} in~$E$ is an infinite sequence $p = e_1 \dots e_n \dots$ of edges in~$E$ such that $r(e_i) = s(e_{i+1})$ for all $i \ge 1$. In this case, we say that the infinite path $p$ starts at the vertex $s(p) \defeq s(e_1)$. We denote by $E^{\infty}$ the set of all infinite paths in~$E$.

The following construction of a groupoid $\mcG_E$ from an arbitrary graph $E$ can be found in \cite[Ex.~2.1]{cs:eghmesa}.
Let $E = (E^0, E^1, r, s)$ be a graph. First let
\[ X_E \defeq \{ p \in E^* \mid r(p) \text{ is a sink or an infinite emitter} \} \cup E^{\infty} . \]
Then the \emph{graph groupoid} associated to~$E$ is defined as
\[ \mcG_E \defeq \{ (\alpha x, |\alpha| \!-\! |\beta|, \beta x) \mid \alpha, \beta \in E^* ,\, x \in X_E ,\, r(\alpha) = s(x) = r(\beta) \} . \]
The formulas $(x, k, y) (y, l, z) = (x, k \!+\! l, z)$ and $(x, k, y)^{-1} = (y, -k, x)$ define composition and inverse maps on $\mcG_E$, making it a groupoid with unit space $\mcG_E^{(0)} = \{( x, 0, x) \mid x \in X_E \}$, which we may identify with the set $X_E$.
Note that the range and source maps $r_{\mcG_E}, s_{\mcG_E} \colon \mcG_E\longrightarrow \mcG_E^{(0)}$ are defined by $r_{\mcG_E}(x, k, y) = (x, 0, x)$ and $s_{\mcG_E}(x, k, y) = (y, 0, y)$, so that we may view each $(x, k, y) \in \mcG_E$ as a morphism with range~$x$ and source~$y$.

We next describe the topology on~$\mcG_E$. For $\alpha, \beta \in E^*$ with $r(\alpha) = r(\beta)$, and a finite subset $F \subseteq s^{-1}(r(\alpha))$, we let
\begin{gather*}
  Z(\alpha, \beta) \defeq \{ (\alpha x, |\alpha| \!-\! |\beta|, \beta x) \mid x\in X_E ,\, r(\alpha) = s(x) = r(\beta) \} \subseteq \mcG_E , \\
  Z(\alpha, \beta, F) \defeq Z(\alpha, \beta) \setminus \bigcup_{e \in F} Z(\alpha e, \beta e) .
\end{gather*}
The sets $Z(\alpha, \alpha, F)$ constitute a base of compact open sets for a locally compact Hausdorff topology on $\mcG_E^{(0)}$ (refer to \cite[Th.~2.1]{w:tpsoadg}, \cite[Th.~2.1]{r:tgatlpa} or \cite[Cor.~2.8]{adn:rcolpastawcctgca}).
And the sets $Z(\alpha, \beta, F)$ constitute a base of compact open bisections for a topology under which
$\mcG_E$ is a Hausdorff ample groupoid (refer to \cite[Sec.~2.3]{bcw:gaaoe} or \cite[Th.~2.4]{r:tgatlpa}).  Thus we may form the Steinberg algebra $A_S(\mcG_E)$. We should note the following properties.

\begin{rem}\label{graphgroupSteinAlg}
Let~$E$ be an arbitrary graph and $S$ a commutative semiring.
\begin{enumerate}
\item The multiplication on $A_S(\mcG_E)$ satisfies the following:
\begin{enumerate}[(i)]
\item $1_{Z(v, v)}\ast 1_{Z(w, w)} = \delta_{v, w}1_{Z(v,v)}$ for all $v, w\in E^0$;
\item $1_{Z(s(e), s(e))}\ast 1_{Z(e, r(e))} = 1_{Z(e, r(e))} = 1_{Z(e, r(e))} \ast 1_{Z(r(e), r(e))}$ for $e \in E^1$;
\item $1_{Z(r(e), r(e))} \ast 1_{Z(r(e), e)} = 1_{Z(r(e), e)} = 1_{Z(r(e), e)} \ast 1_{Z(s(e), s(e))}$ for $e \in E^1$;
\item $1_{Z(r(e), e)}\ast 1_{Z(f, r(f))} = \delta_{e, f}1_{Z(r(e), r(e))}$ for all $e, f\in E^1$;
\item $1_{Z(v, v)} = \sum_{e\in s^{-1}(v)}1_{Z(e, r(e))}\ast 1_{Z(r(e), e)}$ for all regular vertices $v \in E^0$;
\end{enumerate}
where $\delta$ is the Kronecker delta.
\item If in addition~$S$ is additively idempotent, then $A_S(\mcG_E)$ is generated by functions $1_{Z(\alpha, \beta, F)}$, where $\alpha, \beta \in E^*$ with $r(\alpha) = r(\beta)$, and $F \subseteq s^{-1}(r(\alpha))$ is finite. The finite sums of distinct elements of $\{ 1_{Z(v, v)} \mid v \in E^0 \}$ form a set of local units of the $S$-algebra $A_S(\mcG_E)$.
\end{enumerate}
\end{rem}

\begin{proof}
(1) It is straightforward by using Proposition~\ref{convoprod}\,(2).

(2) Lemma~\ref{expresslem}\,(2) immediately implies the first part, from which we easily deduce the remaining part using Proposition~\ref{convoprod}\,(2).
\end{proof}

Our subsequent aim is to characterize the congruence-simpleness of the Steinberg algebra $A_S(\mcG_E)$ over a commutative semiring~$S$. Before doing so, we need some notations and facts.
Let~$E$ be an arbitrary graph. An edge~$f$ is an \emph{exit} for a path $p = e_1 \dots e_n$ if $s(f) = s(e_i)$ but $f \ne e_i$ for some $1 \le i \le n$. A subset~$H$ of~$E^0$ is called \textit{hereditary} if $s(e) \in H$ implies $r(e) \in H$ for all $e \in E^1$. And~$H$ is called \textit{saturated} if whenever~$v$ is a regular vertex in~$E^0$ with the property that $r(s^{-1}(v)) \subseteq H$, then $v \in H$.

The following fact provides us with a criterion for the groupoid~$\mcG_E$ to be effective and minimal.

\begin{prop}\label{eff-minigraphgroup}
Let~$E$ be an arbitrary graph.
\begin{enumerate}[\quad \upshape (1)]
\item $\mcG_E$ is effective if and only if every cycle in $E$ has an exit;
\item $\mcG_E$ is minimal if and only if the only hereditary and saturated subsets of~$E^0$ are $\varnothing$ and $E^0$.
\end{enumerate}
\end{prop}

\begin{proof}
(1) It may be found in \cite[Prop.~2.21]{r:tgatlpa}; and just for the reader's convenience, we reproduce it here. 

($\Longrightarrow$). Suppose that~$E$ has a cycle~$c$ without an exit. We then have that $Z(cc, c) = \{ (ccc \dots, |c|, cc \dots) \} \subseteq \on{Iso}(\mcG_E) \setminus \mcG_E^{(0)}$, and thus~$\mcG_E$ is not effective.

($\Longleftarrow$). Assume that every cycle in~$E$ has an exit, and let $\alpha, \beta\in E^*$ be distinct paths with $r(\alpha) = r(\beta)$. We claim that $\alpha x \ne \beta x$ for some $x \in X_E$.
Indeed, suppose that $\alpha x = \beta x$ for some $x \in X_E$. Then $x \in E^{\infty}$ and one of~$\alpha$ and~$\beta$ is a prefix of the other. We thus may assume that $\beta = \alpha \gamma$ for some closed path $\gamma = e_1 \dots e_k$ based at $r(\alpha) = r(\beta)$.
By our hypothesis and \cite[Lem.~2.5]{ap:tlpaoag05}, the path~$\gamma$ has an exit $f$, \textit{i.e.}, there is $1\le i\le k$ such that $f\ne e_i$ and $s(f) = s(e_i)$. Let~$y$ be an arbitrary element in $X_E$ with $s(y) = r(f)$, then putting $z \defeq e_1 \dots e_{i-1} f y \in X_E$ we find that $\alpha z \ne \beta z$. In the other case, we let~$x$ be an arbitrary element in~$X_E$ with $r(\alpha) = s(x) = r(\beta)$ and have that $\alpha x \ne \beta x$, thus proving the claim.
Hence, there is an element $(\alpha x, |\alpha| \!-\! |\beta|, \beta x) \in Z(\alpha, \beta)$ such that $\alpha x \ne \beta x$. This implies that every compact open bisection $B \subseteq \mcG_E \setminus \mcG_E^{(0)}$ contains an element~$\mu$ such that $s(\mu) \ne r(\mu)$, and so~$\mcG_E$ is effective by \cite[Lem.~3.1]{bcfs:soaateg}. \medskip

(2) ($\Longrightarrow$). Assume that~$\mcG_E$ is minimal, and let~$H$ be a nonempty hereditary and saturated subset of~$E^0$. We prove that $H = E^0$. To this end, let $Z \defeq \bigcup_{v\in H}Z(v, v) \subseteq \mcG_E^{(0)}$. Then $U \defeq s_{\mcG_E}(\mcG_E^Z) = s_{\mcG_E}(r_{\mcG_E}^{-1}(Z))$ is a nonempty open invariant subset of $\mcG_E^{(0)}$ (note that $\varnothing \ne Z\subseteq U$). Since~$\mcG$ is minimal, $U = \mcG_E^{(0)}$.

Suppose that $H\ne E^0$, and let $v \in E^0 \setminus H$. Consider the following cases.

\textit{Case~1:} The vertex~$v$ is a sink or an infinite emitter. Then $(v, 0, v) \in \mcG_E^{(0)} = U$, so there is an element $(y, k, x) \in \mcG_E$ such that $(x, 0, x) = s_{\mcG_E}(y, k, x) = (v, 0, v)$ and $(y, 0, y) = r_{\mcG_E}(y, k, x) \in Z$. This implies that $x = v$ and $y \in E^*$ with $r(y) = v$ and $s(y) \in H$. Since~$H$ is hereditary and $s(y) \in H$, we have $v = r(y) \in H$, a contradiction.

\textit{Case~2:} The vertex~$v$ is regular. Since~$H$ is saturated, there exists an edge $e_1 \in s^{-1}(v)$ such that $r(e_1) \notin H$. If $r(e_1)$ is either a sink or an infinite emitter, then we make a contradiction by repeating the argument described in Case~1, starting with $r(e_1)$, and so we may assume that $r(e_1)$ is a regular vertex.
Continuing this process, we obtain the following possible cases. Either, we arrive after~$n$ steps at a path $p = e_1 \dots e_n$ with $s(p) = v$ and $r(p) \notin H$ is either a sink or an infinite emitter. Then, we produce a contradiction by repeating the argument described in Case~1, starting with $r(p)$.
In the other case, we obtain an infinite path $x = e_1 \dots e_n \dots$ such that $s(x) = v$ and $r(e_n) \notin H$ for all $n \ge 1$. We then have that $(x, 0, x) \in \mcG_E^{(0)}$, \textit{i.e.}, $(x, 0, x) \in U$, and so there exists an element $(z, k, y) \in \mcG_E$ such that $(y, 0, y) = s_{\mcG_E}(z, k, y) = (x, 0, x)$ and $(z, 0, z) = r_{\mcG_E}(z, k, y)\in Z$.
This implies that $y = x$ and $s(z) \in H$. Since $(z, k, x) = (z, k, y) \in \mcG_E$, there exist an integer $n \ge 2$ and a path $\alpha \in E^*$ such that $z = \alpha e_n e_{n+1} \ldots \in E^{\infty}$. Because~$H$ is hereditary, and $s(\alpha) = s(z) \in H$, we find that $r(e_{n-1}) = s(e_n) = r(\alpha) \in H$, which is a contradiction.

In any case, we arrive at a contradiction, and so we infer that $H = E^0$.

($\Longleftarrow$). Suppose that~$E^0$ has only the trivial hereditary and saturated subsets. Let~$V$ be a nonempty compact open subset of $\mcG_E^{(0)}$, and let~$I$ be the ideal of $A_{\B}(\mcG_E)$ generated by~$1_V$. We claim that $I = A_{\B}(\mcG_E)$.
Indeed, since~$V$ is open in $\mcG_E^{(0)}$, there exist a path $\alpha \in E^*$ and a finite subset $F \subseteq s^{-1}(r(\alpha))$ such that $\varnothing \ne Z(\alpha, \alpha, F) \subseteq V$. As $Z(\alpha, \alpha, F)$ is non-empty, there is a path $\beta \in E^*$ that has~$\alpha$ as a prefix path and such that $Z(\beta, \beta) \subseteq Z(\alpha, \alpha, F) \subseteq V$.
Since~$I$ is an ideal of $A_{\B}(\mcG_E)$ and by Proposition~\ref{convoprod}\,(2), we then have $1_{Z(\beta, \beta)} = 1_{Z(\beta, \beta) \cap V} = 1_{Z(\beta, \beta)} \ast 1_V \in I$.
It follows that \[ 1_{Z(r(\beta), r(\beta))} = 1_{Z(r(\beta), \beta) Z(\beta, \beta) Z(\beta, r(\beta))} = 1_{Z(r(\beta), \beta)} \ast 1_{Z(\beta, \beta)} \ast 1_{Z(\beta, r(\beta))} \in I . \]

Let $H \defeq \{ v \in E^0 \mid 1_{Z(v, v)} \in I\}$. We have that $r(\beta) \in H$, and hence $H \ne \varnothing$. Let $e \in E^1$ with $s(e) \in H$. Using Remark~\ref{graphgroupSteinAlg}\,(1), we obtain that \[ 1_{Z(r(e), r(e))} = 1_{Z(r(e), e)} \ast 1_{Z(s(e), s(e))} \ast 1_{Z(e, r(e))} \in I , \] hence $r(e) \in H$, showing that~$H$ is hereditary.

Let~$v$ be a regular vertex such that $r(e) \in H$ for all $e \in s^{-1}(v)$. We then have $1_{Z(r(e), r(e))} \in I$ for all $e \in s^{-1}(v)$, and thus, by Remark~\ref{graphgroupSteinAlg}\,(1), $1_{Z(e, r(e))} = 1_{Z(e, r(e))} \ast 1_{Z(r(e), r(e))}\in I$ for all $e\in s^{-1}(v)$, and
\[ 1_{Z(v, v)} = \sum_{e \in s^{-1}(v)} 1_{Z(e, r(e))} \ast 1_{Z(r(e), e)} \in I . \]
This implies that $v \in H$, and hence~$H$ is saturated. By our hypothesis, $H = E^0$, and hence $1_{Z(v, v)}\in H$ for all $v \in E^0$, whence $I = A_{\B}(\mcG_E)$ by Remark~\ref{graphgroupSteinAlg}\,(2), proving the claim. By Lemma~\ref{minimalcriterion}, $\mcG_E$ is minimal, thus finishing the proof.
\end{proof}

Combining Theorem~\ref{Nec-suffcondtheo} and Proposition~\ref{eff-minigraphgroup}, we readily obtain a complete characterization of the congruence-simple Steinberg algebras $A_S(\mcG_E)$ over commutative semirings.

\begin{thm}\label{con-sim-gragrouSteinAlg}
Let~$E$ be an arbitrary graph and~$S$ a commutative semiring. Then, $A_S(\mcG_E)$ is congruence-simple if and only if the following conditions hold:
\begin{enumerate}[\quad \upshape (1)]
\item $S$ is either a field, or the Boolean semifield~$\B$;
\item The only hereditary and saturated subset of~$E^0$ are~$\varnothing$ and~$E^0$;
\item Every cycle in~$E$ has an exit.
\end{enumerate}
\end{thm}

The remainder of this section is devoted to the study of the connection between Leavitt path algebras and Steinberg algebras. Let us first recall a brief history and the notion of Leavitt path algebras with coefficients in a commutative semiring. Given a row-finite graph~$E$ and any field~$K$, Abrams and Aranda Pino in~\cite{ap:tlpaoag05}, and independently Ara, Moreno, and Pardo in~\cite{amp:nktfga}, introduced the \emph{Leavitt path algebra} $L_K(E)$.
The definition was later generalized to all countable graphs by Abrams and Aranda Pino~\cite{ap:tlpaoag08}, and to all (possibly uncountable) graphs by Goodearl~\cite{g:lpaadl}. Then Tomforde in~\cite{t:lpawciacr} constructed Leavitt path algebras of graphs over a commutative ring, and Katsov and the present authors in~\cite{knz:solpawcias} introduced Leavitt path algebras with coefficients in a commutative semiring.
The notion of a Leavitt path algebra generalizes the algebras $L_K(1, n)$ constructed by Leavitt~\cite{leav:tmtoar} and also encompasses many other interesting classes of algebras. In addition, Leavitt path algebras are intimately related to graph $C^*$-algebras (see \cite{r:ga}).

\begin{defn}[{\cite[Def.~2.1]{knz:solpawcias}}]\label{lpaDef}
Let $E = (E^0, E^1, r, s)$ be an arbitrary graph and let~$S$ be a commutative semiring. The \emph{Leavitt path algebra} $L_S(E)$ of the graph~$E$ with coefficients in~$S$ is the $S$-algebra generated by the union of the set~$E^0$ and two disjoint copies of~$E^1$, say~$E^1$ and $\{ e^* \mid e \in E^1 \}$, satisfying the relations:
\begin{enumerate}
\item $v w = \delta_{v, w} v$ for all $v, w \in E^0$;
\item $s(e) e = e = e r(e)$ and $r(e) e^{\ast} = e^{\ast} = e^{\ast} s(e)$ for all $e \in E^1$;
\item $e^{\ast} f = \delta_{e, f} r(e)$ for all $e, f \in E^1$;
\item $v = \sum_{e \in s^{-1}(v)} e e^{\ast}$ whenever $v \in E^0$ is a regular vertex;
\end{enumerate}
where~$\delta$ is the  Kronecker delta.
\end{defn}

It is easy to see that the mappings given by $v \mapsto v$, for $v \in E^0$, and $e \mapsto e^{\ast}$, $e^{\ast} \mapsto e$ for $e\in E^1$, produce an involution on the algebra $L_S(E)$, and for any path $p = e_1 \dots e_n$ there exists $p^{\ast} \defeq e_n^{\ast} \dots e_1^{\ast}$. For notational convenience we extend the source and range maps by $s(e^*) \defeq r(e)$, $r(e^*) \defeq s(e)$ for all $e \in E^1$, and accordingly $s(p^*) \defeq r(p) = r(e_n)$, $r(p^*) \defeq s(p) = s(e_1)$ for a path $p = e_1 \dots e_n$.

Observe that the Leavitt path algebra $L_S(E)$ can also be defined as the quotient of the free $S$-algebra $S \langle v, e, e^{\ast} \mid v \in E^0 ,\, e \in E^1 \rangle$ by the congruence~$\sim$ generated by the following ordered pairs:
\begin{enumerate}
\item $(v w, \delta_{v, w} v)$ for all $v, w \in E^0$,
\item $(s(e) e, e), (e, e r(e))$ and $(r(e) e^{\ast}, e^{\ast}), (e^{\ast}, e^{\ast} s(e))$ for all $e \in E^1$,
\item $(e^* f, \delta_{e, f} r(e))$ for all $e, f \in E^1$,
\item $(v, \sum_{e \in s^{-1}(v)} e e^{\ast})$ for all regular vertices $v \in E^0$.
\end{enumerate}
 
If~$A$ is an $S$-algebra generated by a family $\{ a_v, b_e, c_{e^{\ast}} \mid v\in E^0 ,\, e\in E^1 \}$ of elements satisfying relations analogous to (1) -- (4) in Definition~\ref{lpaDef}, then there is a unique $S$-algebra homomorphism $\phi \colon L_S(E) \rightarrow A$ given by ${\phi(v) = a_v}$, ${\phi(e) = b_e}$ and ${\phi(e^{\ast}) = c_{e^{\ast}}}$.
We refer to this property as the \textit{universal homomorphism property} of $L_S(E)$. Moreover, by \cite[Prop.~2.4]{knz:solpawcias}, every monomial in $L_S(E)$ is of the form $s p q^*$, where $s \in S$ and $p, q$ are paths in~$E$ such that $r(p) = r(q)$.

Using the universal homomorphism property of $L_S(E)$ and Remark~\ref{graphgroupSteinAlg}\,(1), we immediately obtain that for each graph~$E$ and commutative semiring~$S$, there exists a unique $S$-algebra homomorphism \[ \pi_E \colon L_S(E) \longrightarrow A_S(\mcG_E) \]
such that $\pi_E(v) = 1_{Z(v, v)}$, $\pi_E(e) = 1_{Z(e, r(e))}$, and $\pi_E(e^*) = 1_{Z(r(e), e)}$ for all $v \in E^0$ and $e \in E^1$. In particular, $\pi_E(p q^*) = 1_{Z(p, q)}$ for all paths $p, q \in E^*$ with $r(p) = r(q)$. We refer to this homomorphism as the \textit{natural homomorphism} from $L_S(E)$ to $A_S(\mcG_E)$. Clark and Sims in \cite[Ex.~3.2]{cs:eghmesa} showed that $\pi_E$ is always an isomorphism when~$S$ is a commutative unital ring. However, as the next result shows, this is in general not true for our semiring setting which might be, similarly as for Prop.~\ref{prop:inverse_sg}, explained by a lack of zero sums.

\begin{prop}\label{suj-nathom}
Let $E$ be a graph and $S$ an additively idempotent commutative semiring. Then, the natural homomorphism $\pi_E: L_S(E)\longrightarrow A_S(\mcG_E)$ is surjective if and only if $E$ is row-finite.
\end{prop}

\begin{proof}
($\Longrightarrow$). Suppose that $\pi_E$ is surjective, and let $v \in E^0$. Our claim is that $s^{-1}(v)$ is finite. We may assume that~$v$ is not a sink and choose some nonempty finite subset~$F$ of $s^{-1}(v)$. Since $\pi_E$ is surjective, there exists an element $\alpha \in L_S(E)$ such that $\pi_E(\alpha) = 1_{Z(v, v, F)}$.
By \cite[Prop.~2.4]{knz:solpawcias}, this element can be written in the form $\alpha = \sum^n_{i=1}s_i p_iq^*_i$, where $s_i\in S \!\setminus\! \{ 0 \}$ and $p_i, q_i$ are paths in~$E$ such that $r(p_i) = r(q_i)$. Then we have
\[ 1_{Z(v, v, F)} = \pi_E(\alpha) = \sum_{i=1}^n s_i \pi_E \big( p_i q_i^* \big) = \sum_{i=1}^n s_i 1_{Z(p_i, q_i)} . \]
Since the semiring is zero-sum free, for all $x \in \mcG_E$ we have $\sum_i s_i 1_{Z(p_i, q_i)} (x) \ne 0$ if and only if $x \in \bigcup_i Z(p_i, q_i)$, and therefore \[ Z(v, v, F) = \bigcup_{i=1}^n Z(p_i, q_i) . \]
In particular, $Z(p_i, q_i) \subseteq Z(v, v) \subseteq \mcG_E^{(0)}$ for all~$i$, from which we infer that $p_i = q_i$ and $s(p_i) = s(q_i) = v$. Furthermore, since $F \ne \varnothing$ there holds $p_i = q_i \ne v$ for all~$i$.
Hence, for every $1 \le i \le n$, we may write $p_i = q_i = e_i r_i$ for some edge $e_i \in s^{-1}(v)$ and path~$r_i$ in~$E$, and thus $Z(p_i, q_i) = Z(e_i r_i, e_i r_i) \subseteq Z(e_i, e_i)$. Now since
\[ \bigcup_{e \in s^{-1}(v) \setminus F} Z(e, e) = Z(v, v, F) = \bigcup_{i=1}^n Z(p_i, q_i) \subseteq \bigcup_{i=1}^n Z(e_i, e_i) , \]
we deduce that $s^{-1}(v) \!\setminus\! F \subseteq \{ e_1, \dots, e_n \}$, whence $s^{-1}(v) \subseteq F \cup \{ e_1, \dots, e_n \}$ is a finite set as claimed. Therefore, $E$~is a row-finite graph.

($\Longleftarrow$). Assume that~$E$ is a row-finite graph. Let $\alpha, \beta \in E^*$ with $r(\alpha) = r(\beta)$ and let~$F$ be a finite subset of $s^{-1}(r(\alpha))$. We claim that $\pi_E^{-1}(1_{Z(\alpha, \beta, F)}) \ne \varnothing$. If $r(\alpha)$ is a sink, then necessarily $F = \varnothing$ and $Z(\alpha, \beta, F) = Z(\alpha, \beta)$, so we have that $\alpha \beta^* \in \pi_E^{-1}(1_{Z(\alpha, \beta)})$.
Thus we may assume that $r(\alpha)$ is not a sink, and hence \[ Z(\alpha, \beta, F) = Z(\alpha, \beta) \setminus \bigcup_{e \in F} Z(\alpha e, \beta e) = \bigcup_{e \in F^c} Z(\alpha e, \beta e) , \]
where $F^c \defeq s^{-1}(r(\alpha)) \setminus F$. Since~$E$ is row-finite, $F^c$ is a finite set, and so \[ 1_{Z(\alpha, \beta, F)} = \sum_{e \in F^c} 1_{Z(\alpha e, \beta e)} = \sum_{e \in F^c} \pi_E((\alpha e) (\beta e)^*) = \pi_E \big( \sum_{e \in F^c} (\alpha e) (\beta e)^* \big) , \]
hence $\sum_{f \in F^c} (\alpha e) (\beta e)^* \in \pi_E^{-1}(1_{Z(\alpha, \beta, F)})$, showing the claim. By Remark~\ref{graphgroupSteinAlg}\,(2) this claim implies that~$\pi_E$ is surjective, thus finishing the proof.
\end{proof}

Next we investigate the injectivity of the natural homomorphism $\pi_E$. We note (\cite[Ex.~3.2]{cs:eghmesa}) that if~$S$ is a commutative unital ring, then injectivity follows from the Graded Uniqueness Theorem of Tomforde \cite[Th.~5.3]{t:lpawciacr}, which is based on using graded ring and homogeneous ideal considerations.
In our semiring setting, however, concepts like homogeneous ideal and graded quotient algebra are not well-established, thus we present a novel argument. But first we recall some notations and establish a few useful facts.

Let~$E$ be an arbitrary graph and~$S$ a commutative semiring. Following~\cite{ap:tlpaoag05}, a monomial in $L_S(E)$ is a \emph{real path} if it contains no term of the form $e^{\ast} \in E^{\ast}$, and a polynomial $\alpha \in L_S(E)$ is in \emph{only real edges} if it is an $S$-linear combination of real paths;
let $L_S(E)_{\on{real}}$ denote the subhemiring of all polynomials in only real edges in $L_S(E)$.  For a cycle~$c$ based at the vertex~$v$, we use the notation
\[ c^0 \defeq v \quad\text{ and }\quad c^{-n} = (c^*)^n, \text{ for all $n \in \mathbb N$} . \]
Moreover, for such a cycle~$c$ and any polynomial $p(x) = \sum_{i=m}^n s_i x^i \in S[x, x^{-1}]$ (where $m, n \in \mathbb{Z}$ with $m \le n$), we denote by $p(c)$ the element \[ p(c) \defeq \sum_{i=m}^n s_i c^i \in L_S(E) . \]

The following important fact, being an $\B$-algebra analog of the Reduction Theorem \cite[Th.~2.2.11]{aas:lpa}, provides a method to prove the injectivity of $\B$-algebra homomorphisms from Leavitt path algebras $L_{\B}(E)$ of row-finite graphs~$E$.

\begin{lem}\label{graph-cong-reduction} Let~$E$ be a row-finite graph and~$\rho$ a congruence on $L_{\B}(E)$ different from the diagonal congruence. Then, at least one of the following is true:
\begin{enumerate}[\quad \upshape (1)]
\item $(v, 0) \in \rho$ for some $v \in E^0$;
\item $(p(c), q(c)) \in \rho$, where~$c$ is a cycle in $E$ without exits and $p(x)$, $q(x)$ are distinct polynomials in $\B[x, x^{-1}]$.
\end{enumerate}
\end{lem}

\begin{proof} The proof is essentially based on the ideas in the proof of the direction ($\Longleftarrow$) in \cite[Th.~4.4]{knz:solpawcias}.
	
By \cite[Prop.~4.3]{knz:solpawcias}, the congruence~$\rho$ is generated by $\rho_{\on{real}} \defeq \rho \cap (L_{\B}(E)_{\on{real}})^2$ and $\rho_{\on{real}} \ne \Delta_{L_{\B}(E)_{\on{real}}}$. Hence, there exist two elements $a, b \in L_{\B}(E)_{\on{real}}$ such that $a \ne b$ and $(a, b) \in \rho$.
Since $L_{\B}(E)$ is an additively idempotent hemiring, we can consider the natural order defined by ${s \le s'} \Longleftrightarrow {s+s' = s'}$.  We have $(a, a+b) = (a+a, a+b) \in \rho$, $(b, a+b) = (b+b, a+b) \in \rho$, and since $a \ne b$, either $a < a+b$ or $b < a+b$. Thus, keeping in mind that $(a+x, b+x) \in \rho$ for all $x \in L_{\B}(E)$ and without loss of generality, we may assume that $a < a+b$ and that $a$, $a+b$ are written in the form
\[ a = p_1 + \ldots + p_n, \quad a+b = p_1 + \ldots + p_n + p \]
where $p_1, \dots, p_n, p$ are distinct paths in~$E$. We also may choose~$a$ having the minimal number~$n$ of such $\{p_1 , \dots,  p_n \}$.

Let $v \defeq s(p)$, $w \defeq r(p) \in E^0$. Then $(v a w, v (a+b) w) \in \rho$, where $v a w = v p_1 w + \dots + v p_n w$ and $v (a+b) w = v p_1 w + \ldots + v p_n w + p$, hence by minimality we may assume that $s(p_i) = v$ and $r(p_i) = w$ for all $1 \le i \le n$.

Suppose that $v \ne w$. Write $p = q p'$, where~$q$ is a path from~$v$ to~$w$ of shortest length and~$p'$ is a closed path based at~$w$. For every~$p_j$ such that $q^{\ast} p_j \ne 0$ we have $p_j = q p_j'$ for some closed path~$p_j'$ based at~$w$. Then we have
\[ (q^{\ast} a, q^{\ast} (a+b)) = (q^{\ast} p_1 + \ldots + q^{\ast} p_n, q^{\ast} p_1 + \ldots + q^{\ast} p_n + q^{\ast} p) = (\sum_{j \in J} p_j', \sum_{j \in J} p_j' + p') , \]
and thus $(\sum_{j \in J} p_j', \sum_{j \in J} p_j' + p') \in \rho$ with $p_j'$ (for $j \in J$) and $p'$ distinct closed paths based at~$w$, where~$J$ is a subset of $\{ 1, \dots, n \}$. Therefore, without loss of generality, we may assume that $v = w$, \textit{i.e.}, that $p, p_1, \dots, p_n$ are distinct closed paths based at~$v$, and consider the following two possible cases.

\emph{Case~1:} There is exactly one closed simple path based at~$v$, say $c \defeq e_1 \dots e_m$. It follows that~$c$ is a cycle. Then, there are distinct positive integers~$k$ and $k_i$ for $1 \le i \le n$ such that $p = c^k$ and $p_i = c^{k_i}$ for all~$i$. Write
\begin{gather*}
(c^{\ast})^ka = (c^{\ast})^{h_1} + \ldots + (c^{\ast})^{h_{r}} + c^{h_{r+1}} + \ldots + c^{h_n} \\
(c^{\ast})^k(a+b) = (c^{\ast})^{h_1} + \ldots + (c^{\ast})^{h_{r}} + c^{h_{r+1}} + \ldots+ c^{h_n} + v.
\end{gather*}
If~$c$ has no exit, we may consider distinct polynomials~$p$ and~$q$ in $\B[x, x^{-1}]$, defined by $p(x) = (x^{-1})^{h_1} + \ldots + (x^{-1})^{h_{r}} + x^{h_{r+1}} + \ldots + x^{h_n}$ and $q(x) = (x^{-1})^{h_1} + \ldots + (x^{-1})^{h_{r}} + x^{h_{r+1}} + \ldots+ x^{h_n} + 1$, and deduce that $p(c) = (c^{\ast})^k a$ and $q(c) = (c^{\ast})^k (a + b)$, whence $(p(c), q(c)) \in \rho$, as desired.
On the other hand, if~$c$ has an exit~$f$, \textit{i.e.}, there exists  $1 \le j \le m$ such that $e_j \ne f$ and $s(f) = s(e_j)$, we obtain $(0, r(f)) = (z^{\ast} p(c) z, z^{\ast} q(c) z) \in \rho$ for $z \defeq e_1 \dots e_{j-1} f$, as desired.

\emph{Case~2:} There are at least two distinct closed simple paths
based at~$v$, say~$c$ and~$d$, and we have $c^{\ast} d = 0 = d^{\ast} c$.
Note that $(p^{\ast} a, p^{\ast} (a+b) ) \in \rho$ and let
\begin{gather*}
\alpha \defeq p^{\ast} a = q_1^{\ast} + \ldots + q_{s}^{\ast} + q_{s+1} + \ldots + q_n \\
\beta \defeq p^{\ast}(a+b) = q_1^{\ast} + \ldots + q_{s}^{\ast} + q_{s+1} + \ldots + q_n + v,
\end{gather*}
where $q_1, \dots, q_n$ are closed paths in~$E$ based at~$v$.
Then for some $k \in \mathbb{N}$, where $|c^k| > \max \{ |q_1|, \dots, |q_n| \}$, we get $\alpha' \defeq (c^{\ast})^k x c^k = (c^{\ast})^k q_1^{\ast}c^k + \ldots + (c^{\ast})^k q_s^{\ast} c^k + (c^{\ast})^k q_{s+1}c^k + \ldots + (c^{\ast})^k q_n c^k$ and $\beta' \defeq (c^{\ast})^k y c^k = (c^{\ast})^k q_1^{\ast} c^k + \ldots + (c^{\ast})^k q_s^{\ast} c^k + (c^{\ast})^k q_{s+1}c^k + \ldots + (c^{\ast})^k q_n c^k + v$, and $(\alpha', \beta') \in \rho$.
If $(c^{\ast})^k q_i^{\ast} c^k = 0 = (c^{\ast})^k q_j c^k$ for all $1 \le i \le s$ and $s \!+\! 1 \le j \le n$, then $(0, v) = (\alpha', \beta') \in \rho$. Note that if $(c^{\ast})^k q_j c^k \ne 0$, then $(c^{\ast})^k q_j \ne 0$, and as $|c^k| > |q_j|$, we have $c^k = q_j q_j'$ for some closed path $q_j'$, whence $q_j = c^{\ell}$ for some positive integer $\ell \le k$.
Similarly, in the case $(c^{\ast})^k q_i^{\ast} c^k \ne 0$, we get that $q_i^{\ast} = (c^{\ast})^\ell$ for some positive integer $\ell \le k$.  Since $c^{\ast} d = 0 = d^{\ast} c$, for every $i, j$, one gets $d^{\ast} (c^{\ast})^k q_i^{\ast} c^k d = 0 = d^{\ast} (c^{\ast})^k q_j c^k d$, and hence, $(0, v) = (d^{\ast} \alpha' d, d^{\ast} \beta' d) \in \rho$, as desired, thus finishing the proof.
\end{proof}

Two results of importance, which are direct consequences of Lemma~\ref{graph-cong-reduction}, are the following Uniqueness Theorems.
These results can be considered as the $\B$-algebra analogs of \cite[Th.~5.3, Th.~6.5]{t:lpawciacr}. Namely, the following corollary is an $\B$-algebra analog of the Graded Uniqueness Theorem \cite[Th.~5.3]{t:lpawciacr}.

\begin{cor}\label{Uniqueness Thm}
Let~$E$ be a row-finite graph and~$A$ an arbitrary $\B$-algebra. Then, a hemiring homomorphism $\phi \colon L_{\B}(E) \longrightarrow A$ is injective if and only if the following two conditions are satisfied:
\begin{enumerate}
\item $\phi(v) \ne 0$ for all $v \in E^0$;
\item $\phi(p(c)) \ne \phi(q(c))$ for all cycles~$c$ in~$E$ without exits, and for all distinct polynomials $p(x), q(x) \in \B[x, x^{-1}]$.
\end{enumerate}
\end{cor}	

\begin{proof}
($\Longrightarrow$). Assume that~$\phi$ is injective. Considering the natural homomorphism $\pi_E \colon L_{\B}(E) \longrightarrow A_{\B}(\mcG_E)$, we have $\pi_E(v) = 1_{Z(v, v)} \ne 0$ and so $v \ne 0$, for all $v \in E^0$. Then, since~$\phi$ is injective, $\phi(v) \ne 0$ for all $v \in E^0$, proving condition (1).

Let~$c$ be a cycle in~$E$ without exits, and let $p(x), q(x)$ be two distinct polynomials in $\B[x, x^{-1}]$. It is clear that $p(c), q(c)\in v L_{\B}(E) v$, where $v \defeq s(c) = r(c)$. As shown in the proof of \cite[Prop.~3.1\,(2)]{knz:solpawcias}, the natural homomorphism $\vartheta \colon \B[x, x^{-1}] \to v L_{\B}(E) v$, defined by $\vartheta(f) = f(c)$, is an isomorphism of $\B$-algebras. This implies that $p(c) = \vartheta(p) \ne \vartheta(q) = q(c)$, and hence $\phi(p(c)) \ne \phi(q(c))$, showing condition (2).

($\Longleftarrow$). Suppose that~$\phi$ is not injective. This implies that its congruence on $L_{\B}(E)$, namely $\ker(\phi) \defeq \{ (x, y) \in L_{\B}(E) \mid \phi(x) = \phi(y) \}$, is different from the diagonal congruence. By Lemma~\ref{graph-cong-reduction}, we either have $(v, 0) \in \ker(\phi)$ for some $v \in E^0$, or $(p(c), q(c)) \in \ker(\phi)$, where~$c$ is a cycle in~$E$ without exits, and $p(x)$, $q(x)$ are distinct polynomials in $\B[x, x^{-1}]$.
In the first case, we have $\phi(v) = \phi(0) = 0$, and in the second case, this implies that $\phi(p(c)) = \phi(q(c))$. Thus either condition (1) or condition (2) is violated, finishing the proof.
\end{proof}

The following corollary is an $\B$-algebra analog of the Cuntz-Krieger Uniqueness Theorem \cite[Th.~6.5]{t:lpawciacr}.

\begin{cor}\label{Cuntz-Krieger UniThm}
Let~$E$ be a row-finite graph in which every cycle has an exit, and let~$A$ be an arbitrary $\B$-algebra. If $\phi \colon L_{\B}(E) \longrightarrow A$ is a hemiring homomorphism with $\phi(v) \ne 0$ for all $v \in E^0$, then~$\phi$ is injective.
\end{cor}

\begin{proof}
This follows directly from Corollary~\ref{Uniqueness Thm}, since in the given situation the condition (2) is automatically satisfied.
\end{proof}

Now we are able to prove the following theorem, providing a criterion for the natural homomorphism from the Leavitt path algebra $L_{\B}(E)$ into the Steinberg algebra $A_{\B}(\mcG_E)$ to be an isomorphism. The result establishes that the Leavitt path algebra $L_{\B}(E)$ of a row-finite graph~$E$ is a Steinberg algebra, which can be considered as an $\B$-algebra analog of Clark and Sims's result \cite[Ex.~3.2]{cs:eghmesa}.

\begin{thm}\label{LPAs are Steinberg-Alg}
For any graph~$E$, the natural homomorphism $\pi_E \colon L_{\B}(E) \longrightarrow A_{\B}(\mcG_E)$ is an isomorphism if and only if~$E$ is row-finite.
\end{thm}

\begin{proof}
($\Longrightarrow$). It follows from Proposition~\ref{suj-nathom}.

($\Longleftarrow$). Assume that~$E$ is row-finite. By Proposition~\ref{suj-nathom}, $\pi_E$ is surjective. We claim that $\pi_E$ is injective by using Corollary~\ref{Uniqueness Thm}. Indeed, we first have $\pi_E(v) = 1_{Z(v, v)} \ne 0$ for all $v \in E^0$.

Let~$c$ be a cycle in~$E$ without exits based at the vertex~$v$, and $p(x)$, $q(x)$ two distinct polynomials in $\B[x, x^{-1}]$, therefore $x^k p(x) \ne x^k q(x)$ for all $k \in \mathbb Z$. We choose an integer~$k$ such that $x^k p(x) = \sum_{i \in F} x^i$ and $x^k q(x) = \sum_{i \in G} x^i$ for some distinct finite subsets~$F$ and~$G$ of $\mathbb N$.
Then we have $c^k p(c) = \sum_{i \in F} c^i$ and $c^k q(c) = \sum_{i \in G} c^i$, and hence $\pi_E(c^k p(c)) = \sum_{i \in F} 1_{Z(c^i, v)} = 1_{\bigsqcup_{i \in F} Z(c^i, v)}$ as well as $\pi_E(c^k q(c)) = \sum_{i \in G} 1_{Z(c^i, v)} = 1_{\bigsqcup_{i \in G} Z(c^i, v)}$.
Since $F \ne G$, we have \[ \pi_E(c^k p(c)) = 1_{\bigsqcup_{i \in F} Z(c^i, v)} \ne 1_{\bigsqcup_{i\in G} Z(c^i, v)} = \pi_E(c^k q(c)) , \]
whence $\pi_E(p(c)) \ne \pi_E(q(c))$. From these observations and Corollary~\ref{Uniqueness Thm}, we obtain that $\pi_E$ is injective, proving the claim and finishing the proof.
\end{proof}

In light of Theorem~\ref{LPAs are Steinberg-Alg}, the natural question arises whether there exists any isomorphism between the Leavitt path algebra $L_{\B}(E)$ and the Steinberg algebra $A_{\B}(\mcG_E)$, where~$E$ is an arbitrary graph. The following example gives a negative answer to this question. Before presenting it, we recall the notion of graph inverse semigroups introduced by Mesyan and Mitchell in \cite{mm:tsoagis}. 

Given a graph $E = (E^0, E^1, r, s)$, the \textit{graph inverse semigroup} $G(E)$ of $E$ is the semigroup with zero generated by the sets $E^0$ and $E^1$, together with a set of variables $\{e^{-1}\mid e\in E^1\}$, satisfying the following relations:
\begin{enumerate}
\item $v w = \delta_{v,w} v$ for all $v, w \in E^0$;
\item $s(e) e = e = e r(e)$ and $r(e) e^{-1} = e^{-1} = e^{-1} s(e)$ for all $e \in E^1$;
\item $e^{-1} f = \delta_{e,f} r(e)$ for all $e, f \in E^1$;
\end{enumerate}
where~$\delta$ is the Kronecker delta. We define $v^{-1} = v$ for each $v \in E^0$, and for any path $p = e_1 \dots e_n$ in~$E$ we let $p^{-1} \defeq e_n^{-1} \dots e_1^{-1}$.
With this notation, every nonzero element of $G(E)$ can be written uniquely as $p q^{-1}$ for some paths $p, q \in E^*$. It is not hard to verify that $G(E)$ is indeed an inverse semigroup, with $(p q^{-1})^{-1} = q p^{-1}$ for all $p, q \in E^*$. For further reference we refer to~\cite{mm:tsoagis}.

\begin{exa}\label{LPAs are not Steinberg-Alg}
Let $E = (E^0, E^1, r, s)$ be the graph with a single node $E^0 = \{ v \}$ and countably infinite set of edges $E^1 = \{ e_n \mid n \in \mathbb N \}$, where $r(e_n) = v = s(e_n)$ for all $n \in \mathbb N$. Then, $L_{\B}(E)$ is not isomorphic to $A_{\B}(\mcG_E)$.
\end{exa}

\begin{proof}
Consider the semigroup algebra $\B[G(E)]$. By the universal homomorphism property of $L_{\B}(E)$, there exists a unique $\B$-algebra homomorphism $\vartheta \colon L_{\B}(E) \longrightarrow \B[G(E)]$ such that $\vartheta(p q^*) = p q^{-1}$ for all $p, q \in E^*$. Since $\B[G(E)]$ is the free $\B$-semimodule with basis $\{ p q^{-1} \mid p, q \in E^* \}$, the map is an isomorphism. This implies that $L_{\B}(E)$ is the free $\B$-semimodule with basis $\{ p q^* \mid p, q \in E^* \}$.
	
Suppose that $L_{\B}(E)$ is isomorphic to $A_{\B}(\mcG_E)$, and let $\phi \colon L_{\B}(E) \longrightarrow A_{\B}(\mcG_E)$ be an isomorphism. Since~$v$ is the identity of $L_{\B}(E)$ and $1_{Z(v, v)}$ is the identity of $A_{\B}(\mcG_E)$ (note that $Z(v, v) = \smash{\mcG_E^{(0)}}$), it follows that $\phi(v) = 1_{Z(v, v)}$.

Let~$F$ be a nonempty finite subset of $E^1$. We then have $\phi(x) = 1_{Z(v, v, F)}$ for some nonzero element $x\in L_{\B}(E)$. Since $Z(v, v, F)$ is a proper subset of $Z(v, v)$, we get that $\phi(x) = 1_{Z(v, v, F)} \ne 1_{Z(v, v)} = \phi(v)$ and $\phi(v + x) = \phi(v) + \phi(x) = 1_{Z(v, v)} + 1_{Z(v, v, F)} = 1_{Z(v, v)}= \phi(v)$, hence $x \ne v$ and $v + x = v$.

From \cite[Prop.~2.4]{knz:solpawcias} it follows that~$x$ can be written in the form  $x = \sum_{i=1}^n p_i q_i^*$, where $n \ge 1$,  $p_i, q_i \in E^*$ and $p_k q_k^* \ne v$ for some $1 \le k \le n$. We then have $v = v + x = v + \sum_{i=1}^n p_i q_i^*$.
But since $L_{\B}(E)$ is the free $\B$-semimodule with basis $\{ p q^* \mid p, q \in E^* \}$, an equation of the type $v = v + \sum_{i=1}^n p_i q^*_i$ cannot hold in $L_{\B}(E)$. This shows that $L_{\B}(E)$ is not isomorphic to $A_{\B}(\mcG_E)$, finishing the proof.
\end{proof}

Combining Theorems~\ref{con-sim-gragrouSteinAlg} and \ref{LPAs are Steinberg-Alg} with \cite[Ex.~3.2]{cs:eghmesa}, we readily obtain a complete characterization of the congruence-simple Leavitt path algebras $L_S(E)$ of row-finite graphs over commutative semirings, which was established in \cite[Th.~4.5]{knz:solpawcias} by another approach.

\begin{cor}[{\textit{cf.}~\cite[Th.~4.5]{knz:solpawcias}}] 
Let~$E$ be a row-finite graph and~$S$ a commutative semiring. Then, $L_S(E)$ is congruence-simple if and only if the following three conditions are satisfied:
\begin{enumerate}[\quad \upshape (1)]
\item $S$~is either a field, or the Boolean semifield~$\B$;
\item The only hereditary and saturated subset of~$E^0$ are~$\varnothing$ and~$E^0$;
\item Every cycle in~$E$ has an exit.
\end{enumerate}
\end{cor}

We close this article with the following remark.

\begin{rem}\label{finalrem}
Theorem~\ref{con-sim-gragrouSteinAlg} provides us with nice examples of additively idempotent congruence-simple semirings by using graph groupoids, which may not isomorphic to the corresponding Leavitt path algebras. For example, let~$E$ be the graph as in Example~\ref{LPAs are not Steinberg-Alg}. It is obvious that~$E^0$ has the trivial hereditary and saturated subsets, and every cycle in~$E$ has an exit. Therefore, by Theorem~\ref{con-sim-gragrouSteinAlg}, $A_{\B}(\mcG_E)$ is an additively idempotent congruence-simple infinite semiring. We have not yet known whether $L_{\B}(E)$ is congruence-simple. However, in any case we obtain a congruence-simple semiring being non-isomorphic to $L_{\B}(E)$.
\end{rem}

\end{document}